\documentclass[reqno]{amsart}
\usepackage[scale=0.75, centering, headheight=14pt]{geometry}
\usepackage[latin1]{inputenc}
\usepackage[T1]{fontenc}
\usepackage{lmodern}
\usepackage[english]{babel}
\usepackage{stmaryrd}
\usepackage{tikz}
\usepackage{bbm}
\usepackage{amsmath,amssymb,amsfonts,amsthm}
\usepackage{mathtools,accents}
\usepackage{mathrsfs}
\usepackage{xfrac}
\usepackage{array} 
\usepackage{aliascnt}
\usepackage{booktabs} 
\usepackage{array} 

\usepackage{verbatim} 
\usepackage{subfig} 
\usepackage{mathrsfs}
\usepackage{mathrsfs, dsfont}
\usepackage{amssymb}
\usepackage{amsthm}
\usepackage{amsmath,amsfonts,amssymb,esint}
\usepackage{graphics,color}
\usepackage{enumerate}
\usepackage{mathtools,centernot}

\usepackage{microtype}
\usepackage{paralist} 
\usepackage{cases}
\usepackage[initials]{amsrefs}
\allowdisplaybreaks

\usepackage{braket}
\usepackage{bm}

\usepackage[citecolor=blue,colorlinks]{hyperref}
\addto\extrasenglish{}

\usepackage{enumerate}
\usepackage{xcolor}

\usepackage{aliascnt}

\makeatletter
\def\newaliasedtheorem#1[#2]#3{
  \newaliascnt{#1@alt}{#2}
  \newtheorem{#1}[#1@alt]{#3}
  \expandafter\newcommand\csname #1@altname\endcsname{#3}
}
\makeatother

\theoremstyle{plain}
\newtheorem{theorem}{Theorem}[section]
\newaliasedtheorem{lemma}[theorem]{Lemma}
\newaliasedtheorem{prop}[theorem]{Proposition}
\newaliasedtheorem{claim}[theorem]{Claim}
\newaliasedtheorem{corollary}[theorem]{Corollary}

\theoremstyle{remark}

\theoremstyle{definition}
\newaliasedtheorem{definition}[theorem]{Definition}
\newaliasedtheorem{example}[theorem]{Example}

\newtheorem{OQ(c)}[theorem]{Open Q(c)uestion}
\newtheorem{Q(c)Q(c)}[theorem]{Q(c)uestions}

\theoremstyle{remark}
\newaliasedtheorem{remark}[theorem]{Remark}
\newaliasedtheorem{ex}[theorem]{Example}

\numberwithin{equation}{section}

\def\eps{\varepsilon}

\def\R{\mathbb R}
\def\N{{\mathbb N}}

\DeclareMathOperator{\dv}{div}

\DeclareMathOperator{\curl}{curl}

\DeclareMathOperator{\diam}{diam}

\DeclareMathOperator{\rank}{rank}
\DeclareMathOperator{\loc}{loc}

\DeclareMathOperator{\diag}{diag}

\DeclareMathOperator{\spt}{spt}

\DeclareMathOperator{\dist}{d}

\title{Non-classical 
 solutions of the $p$-Laplace equation}

\author[M. Colombo and  R. Tione]{Maria Colombo \and Riccardo Tione}
\address{Maria Colombo
\hfill\break EPFL B, Station 8, CH-1015 Lausanne, CH}
\email{maria.colombo@epfl.ch}
\address{Riccardo Tione  
\hfill\break  EPFL B, Station 8, CH-1015 Lausanne, CH}
\email{riccardo.tione@epfl.ch}

\begin{document}

\maketitle

\begin{abstract}
In this paper we answer Iwaniec and Sbordone's conjecture \cite{IB94} concerning very weak solutions to the $p$-Laplace equation. Namely, on one hand we show that distributional solutions of the $p$-Laplace equation in $W^{1,r}$ for $p \neq 2$ and $r>\max\{ 1,p-1\}$ are classical weak solutions if their weak derivatives belong to certain cones. On the other hand, we construct via convex integration non-energetic distributional solutions if this cone condition is not met, thus answering negatively Iwaniec and Sbordone's conjecture in general.
\end{abstract}
\par
\medskip\noindent
\textbf{Keywords:} very weak solutions, $p$-Laplace, convex integration.
\par
\medskip\noindent
{\sc MSC (2020): 35D30, 35J15, 35J60, 35J70
\par
}

\section{Introduction}
The $p$-Laplace equation
\begin{equation}
\label{eqn:p-lapl}
\Delta_p u \doteq \dv(|D u|^{p-2} D u) =0 \qquad \mbox{in } \Omega
\end{equation}
which formally corresponds to the Euler-Lagrange equation of the energy
$$\int_\Omega |Du|^p\, dx$$
is one of the most well studied problems in the Calculus of Variations. The classical regularity theory has been achieved in a series of papers of f N. N. Uraltseva \cite{Ur}, K. Uhlenbeck \cite{Uh}, and L. C. Evans \cite{E} for $p\geq 2$, and of J. L. Lewis \cite{L} and P. Tolksdorf \cite{T} for $p>1$  (see also \cite{DB,DBTr84}). To mention some of the milestones obtained for $p$-laplacian-type problems, abandoning the goal to be exaustive, we may cite the counterexamples to regularity of vectorial problems, the Harnack inequality, the partial regularity theory for vectorial problems, the estimates of the singular set, the Calder\'on-Zygmund theory (see \cite{SvYa00
,DuMi04,Li06,DuKrMi07
,Mi11,KuMi18
,Mi06Dark} and the references cited therein). The results for the $p$-Laplace equation have become a paradigm to attack several more complex problems, including the parabolic associated equation, fractional versions of the same equation, functionals with variable exponent $p(x)$, double phase functionals, free boundary problems involving $p$-energies to cite a few.
\\
\\
Given a bounded open set $\Omega \subseteq \R^n$ every distributional solution of the Laplace equation
$$\Delta u=0 \qquad \mbox{in } \Omega$$
is known by Weyl's lemma to be a classical, and therefore also analytic, solution. 
For the $p$-Laplace equation \eqref{eqn:p-lapl}
an analogous result was conjectured by T. Iwaniec and C. Sbordone in \cite{IB94} (see also \cite[Section 9]{Li06}). Indeed the authors conjectured that distributional solutions of the $p$-Laplace equation in $W^{1,r}$ for $p \neq 2$ and $r>\max\{ 1,p-1\}$ (notice that for such solutions one can give the natural distributional meaning to the equation) are of finite energy and hence belong, in the interior of the domain, to $C^{1,\alpha}$ for some $\alpha>0$, according to the classical regularity theory
. This conjecture has been proven true when $r$ is sufficiently close to $p$, namely for $p-\delta<r <p$ for some $\delta$, which depends only on $n$ and $p$.
This result was first obtained in \cite{IB94} via a quantitative version of the Hodge decomposition theorem, previously introduced by Iwaniec \cite{Iw92} in the context of quasiregular mappings (see also \cite{IM93,IM96}). A different approach was then followed by Lewis \cite{Le93}, based on a quantitative version of the Lipschitz truncation.

Our first main result, which is the content of Section \ref{sec:class}, gives a positive answer to the conjecture of \cite{IB94}, under the additional condition \eqref{posbound}:

\begin{theorem}\label{rigidityintro}
Let $p >1$, $\Omega$ be an open and bounded subset of $\R^n$ and $f \in L^1(\Omega)$. Suppose $u \in W^{1,\max\{1,p-1\}}(\Omega)$ solves
\begin{equation}\label{PDE}
\dv(|Du|^{p-2}Du) = f
\end{equation}
in the weak sense on $\Omega$ and, for all $1\le i \le n$, there exist constants $\sigma_i \in \{1,-1\}, L_i \in \R$ such that
\begin{equation}\label{posbound}
\sigma_i\partial_iu \ge L_i \text{ a.e. in $\Omega$}.
\end{equation}
Then $u \in W^{1,p}_{\loc}(\Omega)$ and for all open $\Omega'$ compactly contained in $\Omega$ there exists a constant $C>0$ depending on $n$, $R$, and $\dist(\Omega',\partial\Omega))$, where $R> 0$ is such that $\Omega \subset B_R(0)$, for which
\begin{equation}
\label{ts:rigid}
\int_{\Omega'}|Du|^{p}(x)dx \le C(\|f\|_{L^1(\Omega)} + \|Du\|^{p-1}_{L^{p-1}(\Omega)})(\|u\|_{L^1(\Omega)} + |(L_1,\dots,L_n)|).
\end{equation}

\end{theorem}
Iwaniec and Sbordone's conjecture has a similar flavor of a well understood problem on 
elliptic equations
\begin{equation}
\label{eqn:unif-ell}
\dv(A(x)Du(x)) = 0  \qquad \mbox{in } \Omega,
\end{equation}
with $\lambda I \leq A \leq \Lambda I $. J. Serrin \cite{Se64} provided a striking example, for every $q\in (1,2)$ and $n\geq 2$, of an equation of the form \eqref{eqn:unif-ell} which has an \emph{unbounded} solution. Hence, such solution cannot belong to $ W^{1,2}(\Omega)$ in view of the results of E. De Giorgi and J. Nash. However, in this context subsequent results \cite{HR72,Br08,An09} showed that the situation is completely different as soon as suitable continuity of $A$ is assumed: in this case, every $W^{1,1}$ solution is necessarily of finite energy, namely  $ W^{1,2}(\Omega)$, and the regularity theory applies. 
These results do not provide any clear intuition on the problem we are considering for the $p$-laplacian for two main reasons: on one side, because the freedom in the choice of $A$ is missing in our context, and secondly because it is not clear how to interpret the positive results about the continuous coefficients.
\\
\\
Our second main result, which occupies the rest of the manuscript, shows the sharpness of assumption \eqref{posbound}, see \eqref{dery}, and provides a negative answer to the above-stated conjecture in its full generality.

\begin{theorem}\label{tmain}
Let $\Omega \subset \R^2$ be a ball. For every $p \in (1,\infty), p \neq 2$, there exists $\eps = \eps(p) > 0$ and a continuous $u \in W^{1,p-1 + \eps}(\Omega)$ such that $u$ is affine on $\partial\Omega$,
\begin{equation}\label{dery}
\frac{3}{4} \le \partial_yu \le \frac{5}{4}, \quad \text{a.e. on }\Omega,
\end{equation}
\begin{equation}\label{sol:conv}
\dv(|Du|^{p-2}Du) = 0
\end{equation}
in the sense of distributions, but for all open $B \subset \Omega$
\begin{equation}\label{sol:div}
\int_{B}|Du|^pdx = + \infty.
\end{equation}
\end{theorem}
Fixed $\alpha \in (0,1)$, one may even construct $u \in C^\alpha(\overline{\Omega})$, see \cite[Lemma 2.1]{AFSZ} for details. In order not to add further technical details to the construction, we will content ourselves in proving Theorem \ref{tmain}.
The method we employ to show Theorem \ref{tmain} is the one of convex integration. Specifically, we are going to use the \emph{staircase laminate} construction that, to the best of our knowledge, was introduced by D. Faraco in \cite{MIL}. Since then, this type of constructions has been used in several contexts, see \cite{AFSZ,CFM,FMCO,CFMM,FRT}, and is tailored to tackle problems in which \emph{concentration} phenomena appear. Similar techniques were developed to find counterexamples in which \emph{oscillation} phenomena are the issue. Namely, while the staircase laminate construction deals, roughly speaking, with finding maps which are in some $W^{1,p}$ space but no better, similar methods can be used, for instance, to find maps which Lipschitz and not $C^1$ on any open set, see for instance \cite{SMVS,LSP,COR,JTR,ST}. Instead of outlining our proof here we defer the discussion to Section \ref{STRAT}, where the structure of the part of the paper devoted to the construction of the counterexample will be explained in detail.

\section{When are very weak solutions classical solutions?}\label{sec:class}

This section is devoted to the proof of Theorem~\ref{rigidityintro}, which shows Iwaniec and Sbordone's conjecture assuming the additional condition that each component of $D u$ has a one sided, uniform bound \eqref{posbound}.

\begin{proof}[Proof of Theorem~\ref{rigidityintro}]
Fix $\Omega' \subset \Omega'' \subset \Omega''' \subset \Omega$ compactly contained open sets with 
\begin{equation}\label{dist'}
\dist(\Omega',\partial\Omega) \le 2\dist(\Omega''',\partial\Omega).
\end{equation}
Consider a radial and positive smooth mollification kernel $\rho_\eps$ with support in $B_\eps(0)$, $\forall \eps > 0$ and define the convolution of $u$ with $\rho_\eps$
\[
u_\eps(x) \doteq (u\star \rho_\eps)(x) \qquad \mbox{for every } x\in \Omega'''.
\]
 The latter is well defined as soon as $\eps < \dist(\Omega''',\partial \Omega)$. Observe that \eqref{posbound} still holds for $u_\eps$. 
 
 Let $\psi \in C^\infty_c(\Omega'')$ be such that $\psi(x) \in [0,1]$ for $x \in \Omega$, $\psi \equiv 1$ on $\Omega'$ and 
$\|D\psi\|_{L^\infty(\Omega)} \le {c}{(\dist(\Omega',\partial\Omega))^{-1}}$.
 We test the weak form of \eqref{PDE} with the test function $ u_\eps \psi$, to obtain
\begin{align*}
\sum_{i = 1}^n\int_{\Omega}|Du|^{p-2}(x)\partial_iu(x)\partial_i(u_\eps(x)\psi(x))dx = \int_{\Omega}fu_\eps\psi dx,
\end{align*}
which is equivalent to
\begin{equation}\label{splitalpha}
\sum_{i = 1}^n\int_{\Omega \cap\{ |Du|<1\}}|Du|^{p-2}\partial_iu\partial_iu_\eps\psi dx + \sum_{i = 1}^n\int_{\Omega\cap\{ |Du|\geq1\}}|Du|^{p-2}\partial_iu\partial_iu_\eps\psi dx = \int_{\Omega} (f\psi  + g) u_\eps dx,
\end{equation}
where
\[
g\doteq - \sum_{i = 1}^n|Du|^{p-2}\partial_iu\partial_i\psi.
\]
Notice that the artificial splitting of $\Omega$ in $\Omega \cap\{ |Du|<1\}$ and $\Omega \cap\{ |Du|\geq 1\}$ is necessary since, if $p<2$, one may have $|Du|^{p-2} \notin L^1(\Omega)$.
%
The first term in the left-hand side of \eqref{splitalpha} passes to the limit since $\partial_i u_\eps \to \partial_i u$ in $L^1(\Omega''')$ as $\eps \to 0 $:
\begin{equation}
\label{second:est}
\lim_{\eps\to 0 } \sum_{i = 1}^n\int_{\Omega \cap\{ |Du|<1\}}|Du|^{p-2}\partial_iu\partial_iu_\eps\psi dx =\int_{\Omega \cap\{ |Du|<1\}}|Du|^{p}\psi dx.
\end{equation}
Finally, we exploit \eqref{posbound} by writing
\begin{align*}
\sum_{i = 1}^n\int_{\Omega\cap\{ |Du|\geq1\}}|Du|^{p-2}\partial_iu\partial_iu_\eps\psi dx &= \sum_{i = 1}^n\int_{\Omega\cap\{ |Du|\geq1\}}|Du|^{p-2}(\partial_iu - \sigma_iL_i)(\partial_iu_\eps -\sigma_iL_i)\psi dx 
\\&+ \sum_{i =1}^n\sigma_iL_i\int_{\Omega\cap\{ |Du|\geq1\}}|Du|^{p-2}(\partial_iu_\eps+\partial_iu-\sigma_i L_i)\psi dx.
\end{align*}
Using our choice of $\Omega \cap \{ |Du|\geq1\}$ and $Du \in L^{\max\{1,p-1\}}(\Omega)$, the last  addendum converges to
\[
\sum_{i =1}^n\sigma_iL_i\int_{\Omega\cap\{ |Du|\geq1\}}|Du|^{p-2}(2\partial_iu-\sigma_i L_i)\psi dx.
\]
Moreover, by \eqref{posbound}, we can use Fatou's Lemma to bound from below the first addendum. Therefore, we find
\begin{equation}\label{fatou}
\int_{\Omega'\cap\{ |Du|\geq1\}}|Du|^{p}dx \le \int_{\Omega\cap\{ |Du|\geq1\}}|Du|^{p}\psi dx \le \liminf_{\eps \to 0^+}\sum_{i = 1}^n\int_{\Omega\cap\{ |Du|\geq1\}}|Du|^{p-2}\partial_iu\partial_iu_\eps\psi dx.
\end{equation}
We now estimate the right hand-side of \eqref{splitalpha}. Let $R>0$ be such that $\Omega \subset B_R(0)$. Using the definition of $g$, we find a constant $C' = C'(R,\dist(\Omega',\partial \Omega))$ such that for all $\eps > 0$
\begin{equation}\label{first:est}
\left|\int_{\Omega}(f\psi + g) u_\eps dx\right| \le \|f\|_{L^1(\Omega)}\|u_\eps\|_{L^\infty(\Omega'')} + C'\|Du\|^{p-1}_{L^{p-1}(\Omega)}\|u_\eps\|_{L^\infty(\Omega'')}.
\end{equation}
We claim the existence of a constant $C = C(n,R,\dist(\Omega',\partial \Omega))$ such that for all $0 <\eps < \dist(\Omega''',\partial\Omega)$,
\begin{equation}\label{monotone}
\|u_\eps\|_{L^\infty(\Omega'')} \le C\|u_{\eps}\|_{L^1(\Omega''')} + C|(L_1,\dots, L_n)|.
\end{equation}
Indeed, we define 
\[
L \doteq (\sigma_1L_1,\dots, \sigma_nL_n),\quad u'_{\eps} = u_\eps - (L,x),
\] 
and we show that
\begin{equation}\label{monotone1}
\|u_\eps'\|_{L^\infty(\Omega'')} \le C\|u'_\eps\|_{L^1(\Omega''')},
\end{equation}
which yields \eqref{monotone} for a possibly larger constant $C$. The advantage now is that for all $x \in \Omega'''$ and all $\eps \le \dist(\Omega''',\partial \Omega)$, by \eqref{posbound}
\[
\sigma_i\partial_iu_\eps' \ge 0.
\]
 We consider the set 
\[
S \doteq \left\{h = (h_1,\dots,h_n): \sigma_i h_i \ge 0,\forall i, \max_i\{|h_i|\} \le \frac{\dist(\Omega'',\partial \Omega)}{2}\right\}.
\]
Notice that its volume $|S| = 2^{-n}{\dist(\Omega'',\partial \Omega)^n}
$ is bounded from below by $4^{-n}{\dist(\Omega',\partial \Omega)^n}$. For every $z_0 \in \Omega''$ and $h \in S$,  we estimate $u'_\eps(z_0 - h) \le u'_\eps(z_0) \le u'_\eps(z_0 + h)$.
Integrating on $S$ we obtain
\[
\int_Su'_\eps(z_0 - h)dh\le |S|u'_\eps(z_0) \le \int_Su'_\eps(z_0 + h)dh,
\]
which yields \eqref{monotone}. We now use \eqref{monotone} in \eqref{first:est} to write, for $C$ which is possibly larger than $C$ and $C'$ above,
\begin{equation}\label{first:estdef}
\begin{split}
\limsup_{\eps \to 0}\left|\int_{\Omega}(f\psi +g) u_\eps dx\right| 
&\le C(\|f\|_{L^1(\Omega)} + \|Du\|^{p-1}_{L^{p-1}(\Omega)})(\limsup_{\eps \to 0}\|u_\eps\|_{L^1(\Omega''')} + |(L_1,\dots, L_n)|)
\\&\le C(\|f\|_{L^1(\Omega)} + \|Du\|^{p-1}_{L^{p-1}(\Omega)})(\|u\|_{L^1(\Omega)} + |(L_1,\dots, L_n)|).
\end{split}
\end{equation}
Hence by letting $\eps \to 0$ in \eqref{splitalpha} and combining the lower bound for the left-hand side \eqref{second:est}-\eqref{fatou} and the upper bound for the right-hand side \eqref{first:estdef}, we obtain \eqref{ts:rigid}.
\end{proof}

Before ending this section, we wish to remark a connection between our Theorem \ref{rigidityintro} and the compensation results appearing in the very recent paper \cite{GRS}. A first observation in this regard, which will also be at the basis of the next Sections, is to rewrite \eqref{sol:conv} as a differential inclusion, namely to translate the differential problem \eqref{sol:conv} to that of finding a solution $w \in W^{1,1}(\Omega,\R^2)$ to
\begin{equation}\label{diffinc}
Dw(x) \in K, \quad \text{for a.e. $x \in \Omega$},
\end{equation}
for some suitable set $K \subset \R^{2\times 2}$. The set we will consider is:
\begin{equation}\label{Kp}
K_p \doteq \left\{\left(\begin{array}{cc}x & y \\ |(x,y)|^{p-2}y & -|(x,y)|^{p-2}x \end{array}\right): x,y \in \R\right\}.
\end{equation}
The equivalence between \eqref{sol:conv} and \eqref{diffinc} is achieved once we interpret \eqref{sol:conv} in the equivalent form:
\begin{equation*}\label{divcurl}
0 = \dv(|Du|^{p-2}Du) = \curl(|Du|^{p-2}JDu),\quad \text{where } J = \left(\begin{array}{cc} 0 & 1\\ -1 & 0\end{array}\right).
\end{equation*}
A simple application of Poincar\'e's Lemma yields then the following equivalence:

\begin{prop}\label{prop:translation}
Let $\Omega \subset \R^2$ be convex and $u \in W^{1,\max\{1,p-1\}}(\Omega)$. Then,
\[
\dv(|Du|^{p-2}Du) = 0
\]
in the sense of distributions if and only if there exists $v \in W^{1,1}(\Omega)$ such that $w \doteq (u,v) \in W^{1,1}(\Omega,\R^2)$ solves
\begin{equation}\label{diffincp}
Dw \in K_p\quad \text{ a.e. in $\Omega$}.
\end{equation}
Moreover, for all $q \ge \max\{1,p-1\}$, $u \in W^{1,q}(\Omega)$ if and only if $w \in W^{1,\frac{q}{\max\{1,p-1\}}}$, and there exist positive constants $c_1 < c_2$ such that
\begin{equation}\label{esteq}
c_1\|u\|_{W^{1,q}(\Omega)} \le \|w\|_{W^{1,\frac{q}{\max\{1,p-1\}}}(\Omega,\R^2)} \le c_2\|u\|_{W^{1,q}(\Omega)}.
\end{equation}
\end{prop}

We will sketch how a slightly simplified version of Theorem \ref{rigidityintro} for $n = 2$ can be deduced from \cite[Corollary 4.5]{GRS}. This states that, if
\[
A = A(x) = \left(\begin{array}{cc}a_{11}(x)&a_{12}(x)\\ a_{21}(x) & a_{22}(x)\end{array}\right)  \in C^\infty_c(\R^2,\R^{2\times 2})
\]
with $a_{11} \ge 0$, $a_{22}\le 0$ everywhere on $\R^2$, then
\begin{equation}\label{inbog}
-\int_{\R^2}\det(A) \le \|\curl A_1\|_{L^1(\R^2)}\|\curl A_2 \|_{L^1(\R^2)},
\end{equation}
where $A_i$ denotes the $i$-th row of $A$. Let us define the convex sets
\[
Y = \left\{A = \left(\begin{array}{cc}a_{11}&a_{12}\\ a_{21} & a_{22}\end{array}\right): a_{11}, a_{12}, a_{21}\ge 0, a_{22}\le 0\right\} \subset X = \left\{A = \left(\begin{array}{cc}a_{11}&a_{12}\\ a_{21} & a_{22}\end{array}\right): a_{11}\ge 0, a_{22}\le 0\right\}.
\]
Suppose now that we are given $u \in W^{1,\max\{1,p-1\}}(\Omega)$ satisfying \eqref{PDE} for $f = 0$ and \eqref{posbound} with $\sigma_1 = \sigma_2 = 1$, $L_1 = L_2 = 0$. 
Assumption \eqref{posbound} implies that $Dw \in Y \subset X$, but in order to apply \cite[Corollary 4.5]{GRS} we need to mollify and localize $w$. Consider $\Omega' \subset \Omega'' \subset \Omega$ and a cut-off function $\psi \in C^\infty_c(\Omega'')$ as in the proof above, and define $w_\eps \doteq w\star \rho_\eps$. Finally let
\[
 A_\eps \doteq \psi Dw\star \rho_\eps = D(\psi w_\eps) - w_\eps \otimes D\psi.
\]
We still have that $A_\eps(x) \in Y$ for all $x \in \R^2$, and we estimate, for some $C = C\left(\Omega'',\|Du\|^{\max\{1,p-1\}}_{L^{\max\{1,p-1\}}(\Omega)}\right) > 0$,
\[
\limsup_{\eps \to 0^+}\|\curl A_\eps\|_{L^1(\R^2)} \overset{\eqref{esteq}}{\le} C.
\]
Therefore, we can employ \eqref{inbog} to write
\begin{equation}\label{indet}
-\int_{\Omega}\det(A_\eps)dx \le C.
\end{equation}
Moreover $A_\eps \in Y$ implies $\det(A_\eps) \le 0$ in $\Omega$. Thus we can use Fatou's Lemma to conclude from \eqref{indet} that
\[
-\int_{\Omega}\psi^2 \det(Dw)dx \le C.
\]
As $Dw \in K_p$, we have $\det(Dw) = -|Du|^p$, a.e. in $\Omega$. This gives an alternative proof of Theorem \ref{rigidityintro} in dimension $n=2$.

\section{Convex integration: outline of the strategy of proof of Theorem~\ref{tmain}}\label{STRAT}
The first step in this type of convex integration is to rewrite \eqref{sol:conv} as a differential inclusion as in Proposition~\ref{prop:translation}.
In order to find non-trivial solutions $w$ to \eqref{diffinc}, we can exploit Faraco's staircase laminates. Given an open domain $\omega \subset \R^2$, two distinct $\R^{2\times 2}$ matrices $B,C$ with $\det(B-C) = 0$ and $\lambda  \in (0,1)$, it is possible for all $\eps > 0$ to construct a highly oscillatory Lipschitz and piecewise affine map $f_\eps$ which coincides with any affine map with gradient $A = \lambda B + (1-\lambda)C$ on $\partial \omega$, and such that
\begin{equation}\label{distribution}
Df_\eps \in B_\eps(B)\cup B_\eps(C) \text{ a.e. on $\omega$}
\end{equation}
with a precise estimate on the set of points where $Df_\eps \in B_\eps(B)$ in terms of $\lambda$, see Lemma \ref{sl}. The measure $\mu = \lambda\delta_B + (1-\lambda)\delta_C$ is called a \emph{laminate}. By \emph{splitting} $B$ in another rank-one direction, one obtains a probability measure supported on three (or four) points, which is called a \emph{laminate of finite order} (of order two).  For every laminate of finite order, one can construct a nontrivial family of maps as $f_\eps$ above, see Definition \ref{els}, Lemma \ref{ind}.
\\
\\
A staircase laminate is an element $\mu_n$ of a sequence of laminates of finite order, which is constructed in the following way. Start from a point $A_1 \in \R^{2\times 2}$. First $A_1$ is split into two points, $B_1$ and $E_1$, with $B_1 \in K_p$. $E_1$ is an auxiliary point, which is again split into $C_1 \doteq A_2$ and $D_1 \in K_p$. This yields 
\[
\mu_1 = \lambda_{B_1}\delta_{B_1} + \lambda_{E_1}\lambda_{A_2}\delta_{A_2} + \lambda_{E_1}\lambda_{D_1}\delta_{D_1}.
\]
 Now the \emph{error term} $A_2$ is again split with the same rule through points $B_2 \in K_p$, $D_2 \in K_p$ and a new error term $C_2 \doteq A_3$, which allows us to define $\mu_2$. Inductively, one finds $\mu_n$. In our specific problem, we need $\mu_n$ to fulfill:
\begin{equation}\label{dege}
\int_{\R^{2\times 2}}|X|^{1 + \eps}d\mu_n(X) < + \infty \quad \int_{\R^{2\times 2}}|X|^\frac{p}{p-1}d\mu_n(X) = + \infty,
\end{equation}
and
\begin{equation}\label{error}
\mu_n(\{A_n\}) \to 0, \quad \text{as $n \to \infty$}.
\end{equation}
Conditions \eqref{dege} yield the desired integrability of the solution, while \eqref{error} tells us that the measure $\mu_\infty \doteq \lim_n \mu_n$ is supported  precisely on $K_p$. Unfortunately, as we shall see, this discussion needs some additions to yield an exact solution to the differential inclusion at hand, but it is helpful to understand whether such a laminate can be found in a simpler subset of $K_p$, for instance $K_p\cap \diag(2)$, where $\diag(2)$ denotes the space of $2\times 2$ diagonal matrices. 
\\
\\
We have:
\[
K_p\cap \diag(2) = \left\{\left(\begin{array}{cc}x & 0 \\ 0 & -|x|^{p-2}x \end{array}\right): x,y \in \R\right\}.
\]
Identify diagonal matrices with points of $\R^2$ $(x,y)$, so that for instance every point of $K_p\cap \diag(2)$ is described by the graph $(x,-|x|^{p-2}x)$. Notice that in the space of diagonal matrices rank-one lines are precisely horizontal and vertical lines. Take a sequence of points $A_i \doteq (x_i,y_i)$, with
\[
x_{i + 1} > x_i > 0, \qquad y_{i + 1} > y_i > 0, \quad \forall i \in \N,
\]
to be chosen later. We start by splitting $A_1$ into a vertical direction until we reach a point $B_1 \in K_p \cap \diag(2)$, and the auxiliary point $E_1$. Then, $E_1$ is split horizontally so that it lies in the segment with endpoints $A_2$ and $D_1 \in K_p \cap \diag(2)$. This defines $\mu_1$. Next, given $A_3 = (x_3,y_3)$, we reiterate the reasoning starting from $(x_2,y_2)$. By direct computation, one see that the first condition in \eqref{dege} and \eqref{error} are then equivalent to finding a suitable sequence $\{(x_i,y_i)\}_{i \in \N}$ as above with $x_i,y_i \to \infty$ as $i \to \infty$ such that the following holds, for some small $\eps > 0$:
\[
\lim_{n \to \infty}|(x_n,y_n)|^{1 + \eps} \prod_{k =1}^n\frac{x_{k-1}+y_k^\frac{1}{p-1}}{x_k +y_k^\frac{1}{p-1}}\frac{y_{k-1} + x_{k-1}^{p-1}}{y_k + x_{k-1}^{p-1}} < + \infty.
\]
Notice that for $p= 2$ the latter is impossible, which reflects the validity of Weyl's Lemma. It turns out that the choice $x_i = ai^2$, $y_i = i^{2(p-1)}$ for a suitable positive $a$ that depends on $p$ yields the previous property, and one can directly check that this choice also yields the second condition of \eqref{dege}. $a$ is chosen so that a certain function $G_p(a)$ enjoys particular properties, see Section \ref{choice}. Other sequences $x_i$ and $y_i$ are probably correct, for instance $x_i$ and $y_i$ could be chosen to be of exponential growth, as long as they contain this multiplicative parameter $a$. In our work, the careful choice is exploited in Proposition \ref{ansmag2}, which in turn yields the crucial final estimates of Lemma \ref{upest}.
\\
\\
As said, finding only one staircase laminate is, in general, not enough to finding an exact solution to \eqref{diffinc}. This is due to the fact that the maps constructed through laminates introduce errors in their gradient distribution, see \eqref{distribution}. This is the reason why one considers \emph{in-approximations} of the set $K_p$, using the terminology of \cite{SMVS}, originally due to M. Gromov. Namely, one needs to find a sequence of open sets $U_n$ which \emph{converge} in a suitable sense, see Lemma \ref{conv}, to $K_p$ with the additional property that every point of $U_n$ is the barycenter of a laminate of finite order with suitable integrability properties supported in $U_{n + 1}$. The fact that these sets are open allow one to \emph{absorb} the errors made by property \eqref{distribution}. Once one has a sequence of Lipschitz maps $f_n$ with $Df_n \in U_n$ a.e. which in addition converge strongly in $W^{1,1}$ to a map $f$, it is possible to conclude that $Df \in K_p$ a.e.. The (subsequential) pointwise convergence of the gradients is crucial and is necessary to exploit the convergence of $\{U_n\}$ to $K_p$. It is usually achieved with a mollification trick, see Proposition \ref{induc}\ref{hp3} and the proof of Theorem \ref{tmaintext}.
\\
\\
We will not use the term \emph{in-approximation} in the paper, but our whole effort is precisely based in finding these open sets. The idea is to consider a family of staircase laminates with initial point $P \in Q \subset \R^{2\times 2}$, with $Q$ open. In this case, one obtains endpoints of these laminates of finite order $A_i,B_i,C_i,D_i$, $A_{i + 1} = C_i$, which depend on $P$. The definition of these quantities is in Section \ref{sec:def}. As above, $B_i(P) \in K_p$ and $D_i(P) \in K_p$ for all $P \in Q$, but we will actually need to consider staircase laminates with endpoints lying on the segments connecting $A_i$, $B_i$ and $C_i, D_i$, which do not reach \emph{exactly} $B_i$ and $D_i$, see the definition of the maps $\Phi^k_{i,t}$ of \eqref{phi1}-\eqref{phi2}. This is necessary in order to deal with the errors of \eqref{distribution}. The definition of the laminates and the necessary quantitative estimates is the content of Section \ref{sec:proplam}. In order to find these open sets, we need to prove the openness of the mappings $A_i$ and $\Phi^k_{i,t}$. This is the most technical part of the paper, and occupies Section \ref{open}.
\\
\\
Once openness is shown, we will finally be able to prove the existence of a solution to \eqref{diffincp} in Section \ref{sec:ind}. To deduce that the solution $w$ we construct has the required degeneracy property
\begin{equation}\label{dege2}
\int_{B}|Dw|^\frac{p}{p-1}dx = + \infty
\end{equation}
for all $B \subset \Omega$ open, we will show that our construction yields $w$ not $C^1$ on any open set, and we will deduce \eqref{dege2} from regularity results concerning the $p$-Laplace equation. This is based on Lemma \ref{sigma}, which, together with the aforementioned Lemma \ref{conv} constitutes Section \ref{further}.

\section{Laminates of finite order}\label{LOF}

In this section we introduce elementary splittings and laminates of finite order. We say that $A,B \in\R^{2\times 2}$ are \emph{rank-one connected} if
\[
\rank(A-B) = 1.
\]
We have the following, see for instance \cite[Lemma 2.1]{AFSZ}:
\begin{lemma}\label{sl}
Let $A,B,C \in \R^{2\times 2}$, with $\rank(B - C) = 1$, and $A = \lambda B + (1 - \lambda)C$, for some $\lambda \in [0,1]$. Let also $\Omega \subset \R^2$ be a fixed open domain and $b \in \R^2$. Then, for every $\varepsilon > 0$, one can find a Lipschitz piecewise affine map $f_\eps: \Omega\to \R^2$ such that
\begin{enumerate}
\item $f_\eps(x) = f_0(x) = Ax + b$ on $\partial\Omega$ and $\|f_\eps - f_0\|_{\infty} \le \varepsilon$;
\item $D f_\eps(x) \in B_{\varepsilon}(B)\cup B_\varepsilon(C)$;
\item $|\{x \in \Omega: D f_\eps(x) \in B_{\varepsilon}(B)\}|= \lambda|\Omega|$ and $|\{x \in \Omega: D f_\eps(x) \in B_{\varepsilon}(C)\}|= (1 - \lambda)|\Omega|$.
\end{enumerate}
\end{lemma}
Denote with $\mathcal{P}(U)$ the space of probability measures with support in $U \subset \R^{2\times 2}$.
\begin{definition}\label{els}
Let $\nu,\mu \in \mathcal{P}(U)$, $U \subset  \R^{2\times 2}$ open. Let $\nu = \sum_{i= 1}^N\lambda_i\delta_{A_i}$. We say that $\mu$ can be obtained via \emph{elementary splitting from }$\nu$ if for some $i \in \{1,\dots,N\}$, there exist $B,C \in U$, $\lambda \in [0,1]$ such that
\[
\rank(B-C) = 1,\quad [B,C] \subset U, \quad A_i = sB + (1-s)C,
\]
for some $s \in (0,1)$ and
\[
\mu = \nu +\lambda\lambda_i(-\delta_{A_i} + s\delta_B + (1-s)\delta_C).
\]
A measure $\nu = \sum_{i= 1}^r\lambda_i\delta_{A_i}\in \mathcal{P}(U)$ is called a \emph{laminate of finite order} if there exists a finite number of measures $\nu_1,\dots,\nu_r \in \mathcal{P}(U)$ such that
\[
\nu_1 = \delta_X,\quad \nu_r = \nu
\]
and $\nu_{j + 1}$ can be obtained via elementary splitting from $\nu_j$, for every $j\in \{1,\dots,N-1\}$.
\end{definition}

Iterating Lemma \ref{sl} and using the definition of elementary splitting, see for instance \cite[Lemma 3.2]{SMVS}, one can prove the following:
\begin{lemma}\label{ind}
Let $\Omega \subset \R^2$ be an open domain. Let $U \subset \R^{2\times 2}$ be an open set and let $\nu = \sum_{i = 1}^r\lambda_i\delta_{A_i} \in \mathcal{P}(U)$ be a laminate of finite order with barycenter $A \in \R^{2\times 2}$, i.e.: 
\[
A = \int_{\R^{2\times 2}} Xd\nu(X).
\]
Then, for any $b \in \R^2$ and $\varepsilon>0$, the map $f_0(x)\doteq Ax + b$ admits on $\Omega$ an approximation of piecewise affine, equi-Lipschitz maps $f_\eps \in W^{1,\infty}(\Omega,\R^2)$ with the following properties:
\begin{enumerate}
\item $f_\eps(x) = Ax +b$ on $\partial \Omega$ and $\|f_\eps - Ax - b\|_\infty \le \eps$;
\item $D f_\eps(x) \in \bigcup_{i} B_{\varepsilon}(A_i)$;
\item $|\{x \in \Omega: Df_\eps(x) \in B_\varepsilon(A_i)\}| =\lambda_i|\Omega|,\forall i$.
\end{enumerate}
\end{lemma}

\section{Definition of the main quantities}\label{sec:def}

\subsection{Choice of the sequences and coordinates}\label{choice}

For any $a>0$, consider the following sequence, which is increasing in $i$:
\begin{equation}\label{xy}
x_i = ai^2 + x_0,\quad y_i = i^{2(p-1)} + y_0,\quad \forall i \ge 0.
\end{equation}
 Here $a, x_0$ and $y_0$ will be used as coordinates in our construction. Together with these, we will need a fourth parameter $w \in \R$, as will be clear in the next subsection. The space of parameters is, for any $p > 1$:
\[
Q(c) \doteq (c,2c)\times \left(\frac{3}{4},\frac{5}{4}\right)\times \left(\frac{3}{4},\frac{5}{4}\right)\times \left(\frac{3}{4},\frac{5}{4}\right).
\]
The parameter $c > 0$ will be chosen small if $p > 2$ and large if $p < 2$. Instead of fixing $c$ here, we will update it in various technical results of the next sections. We introduce a function that will be crucial in the following:
\begin{equation}\label{defG}
G_p(a) = \frac{a}{a +  1} + (p-1)\frac{1}{a^{p-1} + 1}.
\end{equation}
We start introducing conditions on $c$: in particular we claim that the inequality
\begin{equation}
\label{eqn:choice-c}
\max\{1,p-1\} < G_p(a) < \max\{1,p-1\} + \frac{1}{2} \qquad \mbox{for every } a\in [c,2c],
\end{equation}
holds for $c$ sufficiently small if $p>2$ and for $c$ sufficiently large if $1<p<2$.

\subsubsection{The case $p > 2$}

The (right) derivative of $c \mapsto G_p(c)$ at $0$ is $1$. Hence, $G_p(c)$ is strictly increasing in a neighborhood of $0$, and, since $G_p(0) = p-1$, for $c$ sufficiently small we have that
\begin{equation}\label{incre}
p-1 < G_p(c) \le G_p(a) \le G_p(2c) < p - \frac{1}{2},  \quad \forall a \in [c,2c].
\end{equation}

\subsubsection{The case $1 <p < 2$}
We have that $\lim_{c \to \infty} G_p(c) = 1$ and that for all $c$ sufficiently large $G_p(c)$ is decreasing. Indeed we have
\[
\frac{d}{dc}G_p(c) =\frac{1}{ (c+1)^2} -(p-1)^2\frac{c^{p-2}}{(c^{p-1} + 1)^2} = \frac{c^{p-2}}{(c^{p-1} + 1)^2}\left(\frac{(c^{p-1} + 1)^2}{c^{p-2}(c + 1)^2} - (p-1)^2\right).
\]
As $1 < p < 2$, $\lim_{c \to \infty}\left(\frac{(c^{p-1} + 1)^2}{c^{p-2}(c + 1)^2} - (p-1)^2\right) = -(p-1)^2 < 0$. Thus, we can take $c$ sufficiently large that
\begin{equation}\label{increless}
1 < G_p(2c) \le G_p(a) \le G_p(c) < \frac{3}{2}, \quad \forall a \in [c,2c].
\end{equation}

\begin{remark}
The function $G_p(a)$ behaves very differently for $p = 2$ and $p \neq 2$. 
 The strict inequalities in \eqref{eqn:choice-c} are the key to be able to conclude the convergence of some quantities in Section \ref{sec:ind}. If $p = 2$, instead, then $G_p(a) \equiv 1$ for all $a$ and our strategy fails, as expected in view of Weyl's Lemma recalled in the introduction.
\end{remark}

\subsection{Parametrization of the laminates}\label{s:par}

Let $p\in(1,\infty)$ be fixed and recall that
\[
K_p = \left\{\left(\begin{array}{cc}x & y \\ |(x,y)|^{p-2}y & -|(x,y)|^{p-2}x \end{array}\right): x,y \in \R\right\}.
\]
We denote points of $Q(c)$ as $P$, but we will almost always drop the dependence of the functions $x_i,y_i,v_i,z_i$, etc., from $P$. For readability, we also introduce the functions
\[
h_w(x) = |(x,w)|^{p-2}w, \quad g_w(x) = |(x,w)|^{p-2}x.
\]
The counterexample is built in an inductive way. As explained in Section \ref{STRAT}, given 
\[
A_i(P) = \left(\begin{array}{cc} x_{i-1} & w\\z_{i-1} & y_{i-1} \end{array}\right),
\]
we need to split it into points $B_i, D_i, C_i = A_{i + 1}$ where $B_i,D_i \in K_p$ and $A_{i + 1}$ is the next step of the iteration. Here, $\{x_i\}$ and $\{y_i\}$ are the sequences of functions introduced in Section \ref{choice}, while $z_i$ cannot be chosen, and will instead be built during the construction of the laminates. We start by considering
\begin{equation}\label{A}
A_i(P) \doteq \left(\begin{array}{cc} x_{i-1} & w\\z_{i-1} & y_{i-1} \end{array}\right).
\end{equation}
Split it into the rank-one direction 
\[
\left(\begin{array}{cc} 0 & 0\\ m & n\end{array}\right),
\]
 for some suitable $m,n \in \R$, and call the endpoints $B_i$ and $E_i$. We impose that $B_i \in K_p$ and $E_i$ has $(2,2)$ component equal to $y_i$. Simple computations show that then $E_i$ and $B_i$ are:
\begin{equation}\label{BE}
B_i(P) \doteq \left(\begin{array}{cc} x_{i-1} & w\\h_{w}(x_{i-1}) & -g_{w}(x_{i-1}) \end{array}\right) \in K_p,\quad  E_i(P) \doteq \left(\begin{array}{cc} x_{i-1} & w\\v_i & y_{i} \end{array}\right),
\end{equation}
where
\begin{equation}\label{vwi}
v_i = \frac{y_i + g_w(x_{i-1})}{y_{i-1} + g_w(x_{i-1})}z_{i-1} - \frac{y_i - y_{i-1}}{y_{i-1} + g_w(x_{i-1})}h_w(x_{i-1}).
\end{equation}
In particular, we can write
\[
A_i =  \lambda_{B_i}B_i+ \lambda_{E_i}E_i,
\]
with
\begin{equation}\label{landa1}
\lambda_{B_i} = \frac{y_i - y_{i-1}}{y_i + g_w(x_{i-1})},\quad \lambda_{E_i} = \frac{y_{i-1} + g_{w}(x_{i-1})}{y_i + g_{w}(x_{i-1})}.
\end{equation}
We further split $E_i$ in the rank-one direction
\[
\left(\begin{array}{cc} m' & 0\\ n' & 0\end{array}\right),
\]
for suitable $m',n' \in \R$, in order to reach the endpoints
\begin{equation}\label{CD}
C_i(P) = A_{i + 1}(P) = \left(\begin{array}{cc} x_{i} & w\\z_{i} & y_{i} \end{array}\right) \text{ and } D_i(P) \doteq \left(\begin{array}{cc}-g^{-1}_w(y_i) & w\\h_w(g^{-1}_w(y_i)) & y_{i} \end{array}\right) \in K_p.
\end{equation}
We have $E_i = \lambda_{C_i}C_i + \lambda_{D_i}D_i$, for
\begin{equation}\label{landa2}
\lambda_{C_i} = \frac{x_{i-1}+g^{-1}_w(y_i)}{x_{i}+g^{-1}_w(y_i)},\quad \lambda_{D_i} = \frac{x_{i} - x_{i-1}}{x_{i}+g^{-1}_w(y_i)}.
\end{equation}
As above, $z_i$ is defined by the requirements $C_i(P) = A_{i + 1}(P)$ and $D_i(P)  \in K_p$:
\begin{equation}\label{fwiimp}
z_i(P) =   \frac{x_i +g^{-1}_w(y_i)}{x_{i-1}+g^{-1}_w(y_i)}v_{i} - \frac{x_i - x_{i-1}}{x_{i-1} +g^{-1}_w(y_i)}h_w(g^{-1}_w(y_i)).
\end{equation}
Notice, though, that we have the freedom to choose $z_0$, that we will fix as a function $z_0 = z_0(P)$ in Subsection \ref{sub:zi}. 
\begin{remark}\label{lof:BDC}
By construction, the probability measure
\[
\nu = \lambda_{B_i(P)}\delta_{B_i(P)} + \lambda_{E_i(P)}\lambda_{D_i(P)}\delta_{D_i(P)}+\lambda_{E_i(P)}\lambda_{C_i(P)}\delta_{C_i(P)}
\]
is a laminate of finite order for all $P \in Q(c)$, see Definition \ref{els}. This fact will be exploited in Section \ref{sec:proplam}.
\end{remark}
To conclude, we introduce maps that will be useful in our constructions. Define for $t \in [0,1]$ and for $P \in Q(c)$,
\begin{equation}\label{phi1imp}
\Phi^1_{i,t}(a,x_0,y_0,w) \doteq A_i(P) + t\lambda_{E_i}(P)(B_i(P) - E_i(P))\in \R^{2\times 2},
\end{equation}
and
\begin{equation}\label{phi2new}
\Phi^2_{i,t}(a,x_0,y_0,w) \doteq E_i(P) + t\lambda_{C_i}(P)(D_i(P) - C_i(P)) \in \R^{2\times 2}.
\end{equation}
Notice that for $t\in (0,1)$ the maps $\Phi^1_{i,t}$ and $\Phi^2_{i,t}$ interpolate between known maps
\begin{equation}\label{endpoints}
\Phi^1_{i,0}(P) = A_i(P),\quad  \Phi^2_{i,0}(P) = E_i(P),\quad \Phi^1_{i,1}(P) = B_i(P),\quad \Phi^2_{i,1} = D_i(P). 
\end{equation}
More explicitely, we rewrite \eqref{phi1imp}-\eqref{phi2new} as:
\begin{equation}\label{phi1}
\Phi^1_{i,t}(a,x_0,y_0,w) \doteq \left(\begin{array}{cc}x_{i-1} &w \\  z_{i-1} + t(h_w(x_{i-1})-z_{i-1}) & y_{i-1} - t(y_{i-1} + g_{w}(x_{i-1}))\end{array}\right),
\end{equation}
\begin{equation}\label{phi2}
\Phi^2_{i,t}(a,x_0,y_0,w) \doteq  \left(\begin{array}{cc} x_{i-1} - t(g^{-1}_w(y_i) + x_{i-1}) & w\\  v_i + t(h_w(g^{-1}_w(y_i))-v_i) & y_i\end{array}\right).
\end{equation}

\section{Openness of the mappings}\label{open}

The goal of this section is to show that the maps $A_i$ and $\Phi_{i,t}^k$ are well-defined, continuous and open. These properties will be shown in Proposition \ref{ansmag2}.

\subsection{Explicit formula for $z_i$}\label{sub:zi}

For all $i \ge 1$, define
\begin{equation}\label{Sell}
S_i(P) \doteq \prod_{\ell = 1}^i\frac{x_\ell +g^{-1}_w(y_\ell)}{x_{\ell-1}+g^{-1}_w(y_\ell)}\frac{y_\ell + g_w(x_{\ell-1})}{y_{\ell-1} + g_w(x_{\ell-1})}
\end{equation}
and for all $\ell \ge 1$,
\begin{equation}\label{Hell}
H_\ell(P) \doteq H^1_\ell(P) + H^2_\ell(P) \doteq \frac{x_\ell +g^{-1}_w(y_\ell)}{x_{\ell-1}+g^{-1}_w(y_\ell)}\frac{y_\ell - y_{\ell-1}}{y_{\ell-1} + g_w(x_{\ell-1})}h_w(x_{\ell-1}) +  \frac{x_\ell - x_{\ell-1}}{x_{\ell-1} +g^{-1}_w(y_\ell)}h_w(g^{-1}_w(y_\ell)).
\end{equation}
We can use fomulas \eqref{vwi}-\eqref{fwiimp}, to find that 
\[
z_i(P) =  \frac{x_i +g^{-1}_w(y_i)}{x_{i-1}+g^{-1}_w(y_i)}\left(\frac{y_i + g_w(x_{i-1})}{y_{i-1} + g_w(x_{i-1})}z_{i-1} - \frac{y_i - y_{i-1}}{y_{i-1} + g_w(x_{i-1})}h_w(x_{i-1})\right) - \frac{x_i - x_{i-1}}{x_{i-1} +g^{-1}_w(y_i)}h_w(g^{-1}_w(y_i)).
\]
Thus, using recursively this formula, we find an expression for $z_i(P)$
\begin{align*}
z_i(P) = S_i(P)z_0 - &\sum_{\ell = 1}^iH_\ell(P)\prod_{k = \ell + 1}^i\frac{x_k +g^{-1}_w(y_k)}{x_{k-1}+g^{-1}_w(y_k)}\frac{y_k + g_w(x_{k-1})}{y_{k-1} + g_w(x_{k-1})} \\
&= S_i(P)z_0 - S_i(P)\sum_{\ell = 1}^{i} \frac{H_\ell(P)}{S_\ell(P)} = S_i(P)\left(z_0 - \sum_{\ell = 1}^{i} \frac{H_\ell(P)}{S_\ell(P)}\right).
\end{align*}
Notice that we work with the convention:
\begin{equation}\label{conven}
\prod_{k = r}^s p_k = 1
\end{equation}
if $r > s$ and $\{p_n\}_n$ is any sequence. We choose $$z_0 = z_0(P) = \sum_{\ell = 1}^{\infty} \frac{H_\ell(P)}{S_\ell(P)}.$$ This yields
\begin{equation}\label{ziexp}
z_i(P) = S_i(P)\sum_{\ell = i + 1}^{\infty} \frac{H_\ell(P)}{S_\ell(P)}.
\end{equation}
This choice requires to show the convergence of $\sum_{\ell = 1}^{\infty} \frac{H_\ell(P)}{S_\ell(P)}$, which will be proved in Proposition \ref{ansmag2} below, after having made some necessary estimates.

\subsection{Sufficient conditions for the openness of the mappings}

We are now interested in providing a sufficient condition to ensure that the mappings $A_{i+1},\Phi^k_{i,t}$ are open, for all $k =1,2, t \in [0,1)$ and $i \ge I$ for $I$ large to be fixed later. To do so, we wish to use the Invariance of Domain theorem, see \cite[Theorem 2B.3]{HAT}, that tells us that we only need to check that the above maps are continuous injections. Continuity is immediate for most of the components of the above maps, and it is only non-trivial for $z_i$ and $v_i$. It will be shown for $z_i$ in Proposition \ref{ansmag2} and from this and \eqref{vwi} it readily follows for $v_i$ too. We only need to verify injectivity. The estimates we need to show in order to prove well-posedness and continuity of $z_i$ and the injectivity of the aforementioned maps will be done in this section.

\begin{lemma}\label{simpl}
Let $Q(c) =  (c,2c)\times \left(\frac{3}{4},\frac{5}{4}\right)\times \left(\frac{3}{4},\frac{5}{4}\right)\times \left(\frac{3}{4},\frac{5}{4}\right)$, with $c$ such that Proposition \ref{ansmag2}\eqref{a1} holds. If $1 < p \neq 2$ and if for some $I \ge 2$ and for all $i \ge I$
\[
(\partial_az_i - i^2 \partial_{x_0}z_i)(P) \doteq \lim_{s\to 0}\frac{z_i(P + s(1,-i^2,0,0))-z_i(P)}{s}
\]
exists, is continuous on $Q(c)$ and
\begin{equation}\label{deriva}
\partial_az_i(P) - i^2\partial_{x_0} z_i(P) > 0, \quad \forall P \in Q(c),
\end{equation}
then, for all $i \ge I, k =1,2$ and $t \in [0,1)$, $A_{i+1},\Phi^k_{i + 1,t}$ are injective mappings and hence $A_{i+1}(Q(c))$ and $\Phi^k_{i + 1,t}(Q(c))$ are open sets.
\end{lemma}

\begin{proof}
By the above discussion, we only need to show the injectivity of the three mappings at hand.
\\
\\
\fbox{$A_{i+1}$:} Let $P_j = (a^j,x_0^j,y_0^j,w^j)$ for $j =1,2$, and suppose that for some index $i \ge I$, 
\begin{equation}\label{PP1}
A_{i+1}(P_1) = A_{i+1}(P_2).
\end{equation} 
We need to show that $P_1 = P_2$. From \eqref{A}, we immediately see that $y_0^1 = y_0^2$ and $w^1 = w^2$. Furthermore, from the equality
\[
x_{i}(P_1) = x_{i}(P_2),
\]
we infer
\begin{equation}\label{interm1}
a^1i^2 + x_0^1 = a^2i^2 + x_0^2 \Leftrightarrow x_0^1 -x_0^2 = -i^2(a^1-a^2).
\end{equation}
Now consider the segment 
\[
\sigma(s) \doteq (a^1,x_0^1,y_0^1,w) + s(a^2 - a^1,x_0^2-x_0^1,0,0) = (a^1,x_0^1,y_0^1,w) + (a^2-a^1)s(1,-i^2,0,0).
\]
As $Q(c)$ is convex, $(\sigma(s),y_0,w) \in Q(c), \forall s \in [0,1]$. Moreover, by assumption \eqref{deriva}, $t\mapsto z_i(\sigma(t))$ is a $C^1$ function. By \eqref{PP1}, we can write
\[
0 = z_i(P_1) -z_i(P_2) = \int_0^1\frac{d}{dt}z_i(\sigma(t))dt = (a^2-a^1)\int_0^1\partial_az_i(\sigma(t)) - i^2 z_i(\sigma(t))dt.
\]
Assumption \eqref{deriva} shows $a^1 = a^2$. Through \eqref{interm1} we also infer $x_0^1 = x_0^2$ and we conclude the proof.
\\
\\
\fbox{$\Phi^1_{i + 1,t}$:}  Fix $t \in [0,1)$. Assume with the same notation as above that $\Phi^1_{i + 1,t}(P_1) =\Phi^1_{i + 1,t}(P_2)$. Then, by \eqref{phi1} we easily infer $w^1 = w^2$ and
\[
a^1i^2 + x_0^1 = a^2i^2 + x_0^2.
\]
Moreover, since $t \neq 1$, equating the $(2,2)$ component of $\Phi^1_{i,t}(P_1)$ and $\Phi^1_{i,t}(P_2)$, we also find $y_0^1 = y_0^2$. Since $x_i(P_1) = x_i(P_2)$ and $w^1 = w^2$, the $(2,1)$ component is now giving us (again since $t \neq 1$),
\[
z_i(P_1) = z_i(P_2),
\]
and we can conclude as in the previous case using \eqref{deriva}.
\\
\\
\fbox{$\Phi^2_{i + 1,t}$:}  Fix $t \in [0,1)$. Assume with the same notation as above that $\Phi^2_{i + 1,t}(P_1) =\Phi^2_{i + 1,t}(P_2)$. Then, by \eqref{phi2} we easily infer $w^1 = w^2$ and $y_0^1 = y_0^2$. With these observations, and the fact that $t \neq 1$, equality of the $(1,1)$ components yields once again
\[
a^1i^2 + x_0^1 = a^2i^2 + x_0^2,
\]
and equality of the $(2,1)$ components yields 
\begin{equation}\label{vPP}
v_{i + 1}(P_1) = v_{i + 1}(P_2).
\end{equation}
Combining these facts with \eqref{vwi}, we easily see that \eqref{vPP} is equivalent to $z_i(P_1) = z_i(P_2)$ and we once again conclude analogously to the first case.
\end{proof}

\subsection{Computation of the derivatives of $z_i$}\label{compder}
From Lemma \ref{simpl}, we know that we need to compute first order derivatives of $z_i$. In particular let
\[
\delta_i = \partial_a - i^2\partial_{x_0}.
\]
Since
\[
z_i(P) = S_i(P)\sum_{\ell = i + 1}^{\infty} \frac{H^1_\ell(P)+H^2_\ell(P)}{S_\ell(P)},
\]
and it is easy to see that $P\mapsto \frac{S_i(P)(H^1_\ell(P)+H^2_\ell(P))}{S_\ell(P)}$ is a smooth mapping, our aim is to estimate from above and below
\begin{equation}
\begin{split}
\delta_i\left(\frac{S_i(P)H_\ell(P)}{S_\ell(P)}\right) &= \delta_i\left(\frac{S_i(P)}{S_\ell(P)}\right)(H^1_\ell(P)+H^2_\ell(P))+ \frac{S_i(P)}{S_\ell(P)}\delta_i (H^1_\ell(P)+H^2_\ell(P))
,
\end{split}
\end{equation}

Using multiple times that $\delta_i \prod_{i = 1}^N f_i g_i^{-1} = \prod_{i=1}^N f_i g_i^{-1} \sum_{i=1}^N (f_i^{-1} \delta_i f_i - g_i^{-1} \delta_ig_i) $, we get
\begin{equation}\label{deltaHell1}
\begin{split}
\delta&_iH^1_\ell(P) =  \delta_i\left(\frac{x_\ell +g^{-1}_w(y_\ell)}{x_{\ell-1}+g^{-1}_w(y_\ell)}\frac{y_\ell - y_{\ell-1}}{y_{\ell-1} + g_w(x_{\ell-1})}h_w(x_{\ell-1}) 
\right)\\
&=H^1_\ell(P) \left(\frac{\delta_ix_\ell}{x_{\ell}+g^{-1}_w(y_\ell)}- \frac{\delta_ix_{\ell - 1}}{x_{\ell-1}+g^{-1}_w(y_\ell)}
+\frac{h'_w(x_{\ell-1})\delta_ix_{\ell - 1}}{h_w(x_{\ell-1})} -\frac{g_w'(x_{\ell-1})\delta_ix_{\ell-1}}{y_{\ell-1} + g_w(x_{\ell-1})}
\right) 
\\
&=H^1_\ell(P) \left(\frac{\delta_i(x_\ell - x_{\ell - 1})}{x_{\ell}+g^{-1}_w(y_\ell)}- \frac{\delta_ix_{\ell - 1}(x_\ell - x_{\ell - 1})}{(x_{\ell}+g^{-1}_w(y_\ell))(x_{\ell-1}+g^{-1}_w(y_\ell))}
+\frac{h'_w(x_{\ell-1})\delta_ix_{\ell - 1}}{h_w(x_{\ell-1})} -\frac{g_w'(x_{\ell-1})\delta_ix_{\ell-1}}{y_{\ell-1} + g_w(x_{\ell-1})}
\right) 
\\&\doteq H^1_\ell(P) \rho^1_{i,\ell} ,
\end{split}
\end{equation}
\begin{equation}\label{deltaHell2}
\begin{split}
\delta_iH^2_\ell(P) &= \delta_i\left(  \frac{x_\ell - x_{\ell-1}}{x_{\ell-1} +g^{-1}_w(y_\ell)}h_w(g^{-1}_w(y_\ell))\right)
= H^2_\ell(P)\left(\frac{\delta_i(x_\ell - x_{\ell-1})}{x_\ell- x_{\ell-1} }- \frac{\delta_ix_{\ell - 1}}{x_{\ell-1} +g^{-1}_w(y_\ell)}\right) \doteq H^2_\ell(P)\rho^2_{i,\ell}.
\end{split}
\end{equation}
%
Similarly, we get
\begin{equation}
\begin{split}
 \delta_i&\left(\frac{S_i(P)}{S_\ell(P)}\right)
=
\delta_i\left(\prod_{k = i + 1}^\ell\frac{x_{k-1}+g^{-1}_w(y_k)}{x_k +g^{-1}_w(y_k)}\frac{y_{k-1} + g_w(x_{k-1})}{y_k + g_w(x_{k-1})}\right)
\\&= \frac{S_i}{S_\ell}\sum_{r = i + 1}^\ell\left[-\frac{\delta_ix_r}{x_{r}+g^{-1}_w(y_r)} + \frac{\delta_ix_{r-1}}{x_{r-1}+g^{-1}_w(y_r)} - \frac{g'_w(x_{r-1})\delta_ix_{r-1}}{y_{r} + g_w(x_{r-1})} + \frac{g'_w(x_{r-1})\delta_ix_{r-1}}{y_{r-1} + g_w(x_{r-1})}\right]
\\
&=\frac{S_i}{S_\ell} \sum_{r = i + 1}^\ell\left[-\frac{\delta_i(x_r-x_{r-1})}{x_{r}+g^{-1}_w(y_r)} + \frac{(x_{r}-x_{r-1})\delta_ix_{r-1}}{(x_{r-1}+g^{-1}_w(y_r))(x_{r}+g^{-1}_w(y_r))}\right.\left.+ \frac{(y_{r} - y_{r-1})g'_w(x_{r-1})\delta_ix_{r-1}}{(y_{r} + g_w(x_{r-1}))(y_{r-1} + g_w(x_{r-1}))}\right]
\\
&\doteq\frac{S_i}{S_\ell}\sigma_{i,\ell}
.\label{sigmaiell}
\end{split}
\end{equation}

Overall, with the newly defined $\rho^1_{i,\ell}, \rho^2_{i,\ell},\sigma_{i,\ell}$ introduced respectively in \eqref{deltaHell1}, \eqref{deltaHell2}, \eqref{sigmaiell}, we obtain that
\begin{equation}
\label{eqn:deltai_expression}
\delta_iz_i = \sum_{\ell = i + 1}^\infty\frac{S_i}{S_\ell}\left[\sigma_{i,\ell}H_\ell +H^1_\ell \rho^1_{i,\ell} +H^2_\ell \rho^2_{i,\ell}\right].
\end{equation}

\subsection{Asymptotic behavior of the terms composing $\delta_iz_i$}\label{estder}
Let us start by estimating $\sigma_{i,\ell}$, see \eqref{sigmaiell}. Throughout this section, we will make a list of claims concerning the asymptotics of the terms involved. We start by claiming that
\begin{align}
&\tag{Asymp. 1}\label{cla1}-\frac{\delta_i(x_r-x_{r-1})}{x_{r}+g^{-1}_w(y_r)} \sim  -\frac{2}{a + 1}\frac{1}{r},\\
&\tag{Asymp. 2}\label{cla2}\frac{x_{r}- x_{r-1}}{(x_{r-1}+g^{-1}_w(y_r))(x_{r}+g^{-1}_w(y_r))} \sim \frac{1}{r^{3}}\frac{2 a}{(a + 1)^2},\\
&\tag{Asymp. 3}\label{cla3}\frac{(y_{r} - y_{r-1})g'_w(x_{r-1})}{(y_{r} + g_w(x_{r-1}))(y_{r-1} + g_w(x_{r-1}))} \sim  \frac{1}{r^{3}}\frac{ 2(p-1)^2 a^{p-2}}{(1 + a^{p-1})^2}.
\end{align}
Let us explain the terminology we used in the previous lines. Given two sequences of functions $\{a_j\}_{j \ge 1}$ and $\{b_j\}_{j \ge 1}$ defined on $\overline{Q(c)}$, the symbol $a_j \sim b_j$ means the following:
\[
\sup_{P \in \overline{Q(c)}}\left|\frac{a_j(P)}{b_j(P)} - 1\right| \to 0,\qquad \text{ as } j \to \infty.
\]
The proofs of $\eqref{cla1} \div \eqref{cla3}$, as well as the following asymptotic estimates, are easy and similar to each other. Therefore, let us only show \eqref{cla1}. Using the definition
\[
|(g_w^{-1}(y_r),w)|^{p-2}g_w^{-1}(y_r) = y_r = r^{2(p-1)} + y_0,
\]
we see that
\[
1 + \frac{y_0}{r^{2(p-1)}} = \left|\left(\frac{g_w^{-1}(y_r)}{r^{2}},\frac{w}{r^2}\right)\right|^{p-2}\frac{g_w^{-1}(y_r)}{r^{2}},
\]
which shows that 
\begin{equation}\label{invg}
g_w^{-1}(y_r) \sim r^{2}.
\end{equation}
We can now prove \eqref{cla1} with the help of \eqref{invg}:
\begin{align*}
\frac{\delta_i(x_r-x_{r-1})}{x_{r}+g^{-1}_w(y_r)}\frac{r(a + 1)}{2} =\frac{r^2 - (r-1)^2}{a + x_0  r^{-2}+ r^{-2}g_w^{-1}(y_r)}\frac{a + 1}{2r} = \frac{(2 - r^{-1})(a+1)}{2(a + x_0  r^{-2}+ r^{-2}g_w^{-1}(y_r))} \sim 1.
\end{align*}
Estimates \eqref{cla1}-\eqref{cla2}-\eqref{cla3} imply that, if we define
\[
\tilde \sigma_{i,\ell} \doteq \sum_{r = i + 1}^\ell\left[-\frac{2}{a + 1}\frac{1}{r}+\frac{(r-1)^2 - i^2}{r^{3}}\frac{2 a}{(a + 1)^2} + \frac{(r-1)^2 - i^2}{r^{3}}\frac{ 2(p-1)^2 a^{p-2}}{(1 + a^{p-1})^2}\right]
\]
and we fix $\eps = \eps(c) > 0$, then if $I = I(c,\eps) \in \N$ is sufficiently large, and $i \ge I$, for all $P \in \overline{Q(c)}$,
\begin{equation}\label{estpreS}
\begin{split}
\left|\sigma_{i,\ell}(P) - \tilde \sigma_{i,\ell}(P)\right| \le C\eps\sum_{r = i + 1}^\ell\left[\frac{1}{r} + \frac{(r-1)^2 - i^2}{r^{3}}\right] \le C\eps\sum_{r = i + 1}^\ell\frac{1}{r}.
\end{split}
\end{equation}
Here and below, $C = C(c) > 0$ is a constant depending solely on $c> 0$ (and $p$, that is fixed). This constant may change line-by-line, but we will always denote it with the same letter. We can further simplify \eqref{estpreS} and $\tilde\sigma_{i,\ell}$. Indeed, by integral comparison, one can see that, for all $\ell \ge i + 1$ and  $\alpha > 0$:
\begin{equation}\label{Euler}
\left|\sum_{k = i + 1}^\ell\frac{1}{k}- \ln\left(\frac{\ell}{i + 1}\right)\right| \le \frac{1}{i + 1},
\end{equation}
\begin{equation}\label{Euleralpha}
\left|\sum_{k = i + 1}^\ell\frac{1}{k^{\alpha + 1}}- \frac{1}{\alpha}\left[\frac{1}{(i+1)^\alpha} - \frac{1}{\ell^\alpha}\right]\right| \le \frac{1}{(i + 1)^{\alpha+1}} .
\end{equation}
Furthermore, letting 
\begin{equation}\label{sigmabar}
\begin{split}
\bar \sigma_{i,\ell} &\doteq -\frac{2}{a + 1}\ln\left(\frac{\ell}{i + 1}\right)+\left[\ln\left(\frac{\ell}{i + 1}\right) +\frac{i^2}{2}\left( \frac{1}{ \ell^2} - \frac{1}{(i + 1)^2}\right)\right]\left[\frac{2 a}{(a + 1)^2} +\frac{ 2(p-1)^2 a^{p-2}}{(1 + a^{p-1})^2} \right],
\end{split}
\end{equation}
and using \eqref{Euler}-\eqref{Euleralpha}, we see that, if $I$ is sufficiently large, \eqref{estpreS} becomes, for all $P \in \overline{Q(c)}$, $i \ge I$,
\begin{equation}\label{estSbar}
|\sigma_{i,\ell}(P) - \bar\sigma_{i,\ell}(P)| \le \left|\sigma_{i,\ell}(P) - \tilde \sigma_{i,\ell}(P)\right| + |\tilde\sigma_{i,\ell}(P) - \bar\sigma_{i,\ell}(P)| \le C\left[\eps \ln\left(\frac{\ell}{i + 1}\right) +\frac{1}{i}\right] \le C\eps\left[\ln\left(\frac{\ell}{i + 1}\right) +1\right]
\end{equation}
Let us now move to $H_\ell^1$ and $H_\ell^2$:
\begin{align*}
&\tag{Asymp. 4}\label{cla4}H^1_\ell(P) = \frac{x_\ell +g^{-1}_w(y_\ell)}{x_{\ell-1}+g^{-1}_w(y_\ell)}\frac{y_\ell - y_{\ell-1}}{y_{\ell-1} + g_w(x_{\ell-1})}h_w(x_{\ell-1}) \sim 1 \cdot  \frac{2(p-1)}{1 + a^{p-1}}\frac{1}{\ell}\cdot a^{p-2}\ell^{2(p-2)}w
\doteq \bar H^1_\ell(P).
\end{align*}
\begin{align*}
&\tag{Asymp. 5}\label{cla5}
H^2_\ell(P) = 
 \frac{x_\ell - x_{\ell-1}}{x_{\ell-1} +g^{-1}_w(y_\ell)}h_w(g^{-1}_w(y_\ell))\sim \frac{2a}{a + 1}\frac{1}{\ell} \cdot \ell^{2(p-2)}w\doteq \bar H^2_{\ell}(P).
\end{align*}
Hence we have
\begin{equation}\label{estH}
\left|H^1_\ell(P) - \bar H^1_\ell(P)\right|+\left|H^2_\ell(P) - \bar H^2_\ell(P)\right| \le C\eps \ell^{2(p-2) -1},
\end{equation}
for all $\ell \ge i + 1$, provided $i \ge I$. Define also $\bar H_\ell \doteq \bar H^1_\ell + \bar H^2_\ell$. We turn now to $\rho^1_{i,\ell}$ and $\rho^2_{i,\ell}$. Similarly to what we have done for the previous terms, we can compute the asymptotic behavior of each of the addenda in the definition of  $\rho^1_{i,\ell},\rho^2_{i,\ell}$, given in \eqref{deltaHell1}, \eqref{deltaHell2}. Then, we introduce the following functions
\begin{align*}
\bar \rho^1_{i,\ell} &\doteq \frac{2}{(1 + a)\ell} 
 - \frac{2a((\ell - 1)^2 - i^2)}{(a + 1)^2\ell^3} 
+ \frac{(p-2)((\ell - 1)^2 - i^2)}{a\ell^2} - \frac{(p-1)a^{p- 2}}{(1 + a^{p-1})}\frac{((\ell - 1)^2 - i^2)}{\ell^2} 
,
\end{align*}
\begin{align*}
\bar \rho^2_{i,\ell} &\doteq 
  \frac{1}{a }- \frac{1}{1 + a}\frac{((\ell - 1)^2 - i^2)}{\ell^2},
\end{align*}
in which every addendum is given by the asymptotic behavior of the corresponding term of $\rho_{i,\ell}^1$ and $\rho_{i,\ell}^2$. Then, choosing $I$ possibly larger and if $i \ge I$, we have for all $P \in \overline{Q(c)}$:
\begin{equation}\label{rhoest}
|\rho^1_{i,\ell}(P) - \bar \rho^1_{i,\ell}(P)| +|\rho^2_{i,\ell}(P) - \bar \rho^2_{i,\ell}(P)| \le C\eps
\left(1 +\frac{(\ell-1)^2 - i^2}{\ell^2}\right) \le C\eps.
\end{equation}
\subsection{Asymptotic behavior of $\frac{S_i}{S_\ell}$}\label{ASYSIL}
{\it We claim that for all $p\neq 2$, $p \in (1,+\infty)$, if $c$ fulfills \eqref{incre}-\eqref{increless}, for all $\eps > 0$, there exists $I = I(\eps,c)$ such that
\begin{equation}\label{SIL}
\sup_{P \in \overline{Q(c)}}\left|\frac{S_i(P)}{S_\ell(P)}\left(\frac{\ell}{i + 1}\right)^{2 G_p(a)} - 1\right| \le \eps \qquad \mbox{for all } \ell -1 \ge i \ge I.
\end{equation}
}
We let, for all $p \in (1,+\infty)$,
\begin{equation}\label{def:Sbar}
\bar S_{i,\ell} =\bar S_{i,\ell}(a)  \doteq  \left(\frac{i + 1}{\ell}\right)^{2 G_p(a)}.
\end{equation}
Indeed, fix $\eps > 0$. We have, for $\ell \ge i + 1$:
\begin{equation}\label{prodfund}
\begin{split}
\frac{S_i}{S_\ell}(P) &= \prod_{k = i + 1}^\ell\frac{x_{k-1}+g^{-1}_w(y_k)}{x_k +g^{-1}_w(y_k)}\frac{y_{k-1} + g_w(x_{k-1})}{y_k + g_w(x_{k-1})} \\
&= e^{\sum_{k = i + 1}^\ell(\ln(x_{k-1}+g^{-1}_w(y_k))- \ln(x_{k}+g^{-1}_w(y_k))}e^{\sum_{k = i + 1}^\ell(\ln(y_{k-1}+g_w(x_{k-1}))- \ln(y_{k}+g_w(x_{k-1}))}.
\end{split}
\end{equation}
Using a first order Taylor expansion and the concavity of the logarithm, we see that
\begin{equation}\label{log1}
\frac{x_{k-1} - x_k}{x_{k-1} + g_w^{-1}(y_k)} \le \ln(x_{k-1}+g^{-1}_w(y_k))- \ln(x_{k}+g^{-1}_w(y_k)) \le \frac{x_{k-1} - x_k}{x_{k} + g_w^{-1}(y_k)}.
\end{equation}
We claim that the first and last terms in \eqref{log1} are both $-\frac{2a}{a + 1}\frac{1}{k} $ up to lower order corrections, namely that
\begin{equation}
\label{eqn:claim-main}
\Big|\frac{x_{k-1} - x_k}{x_{k-1} + g_w^{-1}(y_k)}+\frac{2a}{a + 1}\frac{1}{k} \Big|+ \Big|\frac{x_{k-1} - x_k}{x_{k-1} + g_w^{-1}(y_k)}+ \frac{2a}{a + 1}\frac{1}{k} \Big| \leq \frac{C}{k^{1 + \min\{1,2(p-1)\}}}.
\end{equation}
We estimate the first term in the left-hand side of \eqref{eqn:claim-main}; the second term can be treated similarly or by comparing it with the first term. First, by the implicit function theorem, the function $w \mapsto g_w^{-1}(y)$ is $C^1$ and
\[
\partial_w (g_w^{-1}(y)) = -(p-2)\frac{wg_w^{-1}(y)}{w^2 + (p-1)(g_w^{-1}(y))^2}, \quad \forall w,y \in \R.
\]
Therefore, $\forall P \in \overline{Q(c)}$, by \eqref{invg},
\begin{equation}\label{diffgw}
|g_w^{-1}(y_k) - g_0^{-1}(y_k)| \le \frac{C}{k^2}, \quad \forall k \ge 1.
\end{equation}
We use that
$g_0^{-1}(y) = y^\frac{1}{p - 1},$
for all $y \in \R^+$, as follows from the definition of $g_w(x)$, to estimate
\begin{equation}
g_0^{-1}(y_k)-k^2=
 (k^{2(p-1) } + y_0)^\frac{1}{p-1}-{k^{2(p-1)}}^\frac{1}{p-1} \leq 
\begin{cases}
y_0^{\frac{1}{p-1}} \qquad &\mbox{ if }p>2
\\
Ck^{2(2-p)}\qquad &\mbox{ if }p<2.
\end{cases}
\end{equation}
Hence, we write for all $i + 1 \le k \le \ell$:
\begin{align*}
\left|\frac{x_{k-1} - x_k}{x_{k-1} + g_w^{-1}(y_k)} + \frac{2a}{a + 1}\frac{1}{k}\right| &= a \left|\frac{(-2k + 1)(a + 1)k+2a({k-1})^2+2x_0 + 2g_w^{-1}(y_k)}{(x_{k-1} + g_w^{-1}(y_k))(a + 1)k} \right| \\
&\le a \frac{|-3ak+k+2a+2x_0| + 2|g_w^{-1}(y_k)-g_0^{-1}(y_k)|+ 2|g_0^{-1}(y_k)-k^2|}{(x_{k-1} + g_w^{-1}(y_k))(a + 1)k} \\
&\leq C  \frac{k + k^{\max\{0, 2(2-p)\}}}{k^3} \\
&\le \frac{C}{k^{1 + \min\{1,2(p-1)\}}}.
\end{align*}
This proves the estimate of the first term in our claim \eqref{eqn:claim-main}. With \eqref{eqn:claim-main} at hand we obtain
\begin{equation}\label{log12}
\begin{split}
&\left|\sum_{k = i + 1}^\ell(\ln(x_{k-1}+g^{-1}_w(y_k))- \ln(x_{k}+g^{-1}_w(y_k)) + \frac{2a}{a + 1}\ln\left(\frac{i + 1}{\ell}\right)\right|
\\&\le 
\left|\sum_{k = i + 1}^\ell\left(\ln(x_{k-1}+g^{-1}_w(y_k))- \ln(x_{k}+g^{-1}_w(y_k))+\frac{2a}{(a + 1)k} \right) \right| + \left|-\sum_{k = i + 1}^\ell\frac{2a}{(a + 1)k} +\frac{2a}{a + 1}\ln\left(\frac{i + 1}{\ell}\right)\right|
\\ 
&\overset{\eqref{eqn:claim-main}-\eqref{Euler}}{\le} \sum_{k = i + 1}^\ell \frac{C}{k^{1 + \min\{1,2(p-1)\}}} + \frac{C}{i + 1} \overset{\eqref{Euleralpha}}{\le} \frac{C}{(i+1)^{ \min\{1,2(p-1)\}}}.
\end{split}
\end{equation}
Analogously, we can prove that
\begin{equation}\label{log13}
\left|\sum_{k = i + 1}^\ell(\ln(y_{k-1}+g_w(x_{k-1}))- \ln(y_{k}+g_w(x_{k-1})) + (p-1)\frac{2}{a^{p-1}+ 1}\ln\left(\frac{i + 1}{\ell}\right)\right| \le \frac{C}{(i + 1)^{\min\{1,2(p-1)\}}}.
\end{equation}
Estimates \eqref{log12}-\eqref{log13} combined with \eqref{prodfund} and the definition of $G_p(a)$, see \eqref{defG}, prove \eqref{SIL}.

\subsection{Estimates of $\delta_iz_i$}We now use our previous asymptotic estimates to bound $\delta_iz_i$ in \eqref{eqn:deltai_expression}. We will observe that the closeness estimates \eqref{estSbar}, \eqref{estH}, \eqref{rhoest}, \eqref{SIL} of $\sigma_{i,\ell}$ and $\bar \sigma_{i,\ell} $, $\frac{S_i}{S_\ell}$ and $\bar S_{i,\ell}$, $H^i_\ell$ and $\bar H^i_{\ell}$ $\rho^j_{i,\ell}$ and $\bar\rho^j_{i,\ell}$ put us in the position to neglect the contributions of $\sigma_{i,\ell}-\bar \sigma_{i,\ell} $, $\frac{S_i}{S_\ell}-\bar S_{i,\ell}$, $H^i_\ell-\bar H^i_{\ell}$ $\rho^j_{i,\ell}-\bar \rho^j_{i,\ell}$ in the computation of  $\delta_iz_i$ in \eqref{eqn:deltai_expression}. Indeed, these estimates imply that the growth of  $\sigma_{i,\ell}$ is the same of the growth of  $\bar \sigma_{i,\ell} $, and similarly for the other quantities. We have that
\begin{align*}
\left[\right.&\left.\sigma_{i,\ell}H_\ell +H^1_\ell \rho^1_{i,\ell} +H^2_\ell \rho^2_{i,\ell}\right] - \left[\bar\sigma_{i,\ell}\bar H_\ell +\bar H^1_\ell \bar \rho^1_{i,\ell} +\bar H^2_\ell \bar\rho^2_{i,\ell}\right]
\\
&=  (\sigma_{i,\ell}- \bar\sigma_{i,\ell})H_\ell +H^1_\ell (\rho^1_{i,\ell}-\bar \rho^1_{i,\ell}) +H^2_\ell (\rho^2_{i,\ell}-\bar \rho^2_{i,\ell}) + \bar\sigma_{i,\ell}(H_\ell -\bar H_\ell)+(H^1_\ell-\bar H^1_\ell) \bar \rho^1_{i,\ell} +(H^2_\ell - \bar H^2_\ell) \bar \rho^2_{i,\ell}.
\end{align*}
Thus, by the triangle inequality, \eqref{estSbar}-\eqref{estH}-\eqref{rhoest} yield, uniformly in $P \in \overline{Q(c)}$:
\begin{equation}
\label{tbsimp0}
\left|
\left[\sigma_{i,\ell}H_\ell +H^1_\ell \rho^1_{i,\ell} +H^2_\ell \rho^2_{i,\ell}\right] - \left[\bar\sigma_{i,\ell}\bar H_\ell +\bar H^1_\ell \bar \rho^1_{i,\ell} +\bar H^2_\ell \bar\rho^2_{i,\ell}\right]
\right| \le C\eps \left(1 + \ln\left(\frac{\ell}{i+1}\right)\right)\ell^{2(p-2) -1}.
\end{equation}
We notice moreover that, using the definitions of $\bar \sigma_{i,\ell}, \bar H_\ell$ and $\bar \rho_{i,\ell}$, we have
\begin{equation}\label{aboveest}
|\bar\sigma_{i,\ell}\bar H_\ell + \bar H^1_\ell \bar \rho^1_{i,\ell} +\bar H^2_\ell \bar\rho^2_{i,\ell}| \le |\bar\sigma_{i,\ell}\bar H_\ell|+ |\bar H^1_\ell \bar \rho^1_{i,\ell} +\bar H^2_\ell \bar\rho^2_{i,\ell}| \le C\left(1 + \ln\left(\frac{\ell}{i+1}\right)\right)\ell^{2(p-2)-1}.
\end{equation}
Similarly, exploiting \eqref{SIL}-\eqref{tbsimp0}-\eqref{aboveest}, we can estimate uniformly in $P \in \overline{Q(c)}$:
\begin{equation}
\label{tbsimp1}
\left|
\frac{S_i}{S_\ell}\left[\sigma_{i,\ell}H_\ell +H^1_\ell \rho^1_{i,\ell} +H^2_\ell \rho^2_{i,\ell}\right] - \bar S_{i,\ell}\left[\bar\sigma_{i,\ell}\bar H_\ell +\bar H^1_\ell \bar \rho^1_{i,\ell} +\bar H^2_\ell \bar\rho^2_{i,\ell}\right]
\right| \le C\eps \bar S_{i,\ell}\left(1 + \ln\left(\frac{\ell}{i+1}\right)\right)\ell^{2(p-2) -1}.
\end{equation}
We need now to divide the cases $p > 2$ and $1 < p < 2$. We will show that for all $c$ sufficiently small if $p > 2$ and $c$ sufficiently large if $p < 2$, and for a possibly larger $I$,
\begin{equation}\label{conc}
\left(\sum_{\ell = i + 1}^\infty\frac{S_i}{S_\ell}(\sigma_{i,\ell}H_\ell+ H^1_\ell \rho^1_{i,\ell} +H^2_\ell \rho^2_{i,\ell})\right)(P) \ge C'i^{2(p-2)-1},
\end{equation}
for all $i + 1 \ge I$, for a positive constant $C' > 0$. Before dividing the two cases, we use integral comparison to find that, for all $i \ge 2$ and $a \in [c,2c]$, with $c$ chosen such that \eqref{incre}-\eqref{increless} hold:
\begin{equation}\label{tail}
\left|\sum_{\ell = i + 1}^\infty \bar S_{i,\ell} \ln\left(\frac{\ell}{i + 1}\right)\ell^{2(p-2) - 1} - \frac{(i + 1)^{2(p-2)}}{4 (G_p(a) - p + 2)^2}\right| \le 2\ln(i + 1)(i + 1)^{2(p-2)-1}
\end{equation}
\begin{equation}\label{tail1}
\begin{split}
\left|\sum_{\ell = i + 1}^\infty \bar S_{i,\ell}\left( \frac{i^2}{(i + 1)^2} -  \frac{i^2}{\ell^2}\right)\ell^{2(p-2)-1} \right.&\left.- \left[\frac{1}{2(2-p) + 2G_p(a)} - \frac{1}{2(2-p) + 2G_p(a) + 2}\right](i + 1)^{2(p-2)}\right|\\
&\le 10(i+1)^{2(p-2) - 1}
\end{split}
\end{equation}
\begin{equation}\label{tail2}
\begin{split}
\left|\sum_{\ell = i + 1}^\infty \bar S_{i,\ell} \ell^{2(p-2)-1}\frac{(\ell - 1)^2 - i^2}{\ell^2} \right.&\left.- \left[\frac{1}{2(2-p) + 2G_p(a)} - \frac{1}{2(2-p) + 2G_p(a) + 2}\right](i + 1)^{2(p-2)}\right|\\
&\le 10(i+1)^{2(p-2) - 1}.
\end{split}
\end{equation}
\\
\\
\fbox{The case $p > 2$:} according to \eqref{eqn:choice-c}-\eqref{Euleralpha}-\eqref{tbsimp1} and \eqref{tail}, to show \eqref{conc} it is sufficient to prove
\begin{equation}\label{concc}
\sum_{\ell = i + 1}^\infty\bar S_{i,\ell}( \bar\sigma_{i,\ell}\bar H_{\ell} + \bar H^1_\ell \bar \rho^1_{i,\ell} +\bar H^2_\ell \bar\rho^2_{i,\ell})(P) \ge C'i^{2(p-2)}.
\end{equation}
We see that there exists a constant $L > 0$ (that may change line-by-line) independent of $c$ once we choose for instance $c < 1$, such that for all $\ell \ge i + 1 \ge 2$, $P \in Q(c)$,
\begin{equation}\label{Hsig}
|\bar H_\ell \bar \sigma_{i,\ell}|(P) \le c^{\min\{1,p-2\}}L\left(1 + \ln\left(\frac{\ell}{i + 1}\right)\right)\ell^{2(p-2)-1}.
\end{equation}
On the other hand, if $m = \min_{\overline{Q(c)}}\frac{2 w}{a + 1} > 0$,
\begin{equation}\label{Hsigro}
(\bar H^1_\ell \bar \rho^1_{i,\ell} +\bar H^2_\ell \bar\rho^2_{i,\ell})(P) \ge  \frac{2 w}{a + 1}\ell^{2(p-2)-1} - L c^{\min\{1,p-1,2(p-2)\}}\ell^{2(p-2)-1} \ge (m - L c^{\min\{1,2(p-2)\}})\ell^{2(p-2)-1}.
\end{equation}
Combining \eqref{Euleralpha}-\eqref{tail}-\eqref{Hsig}-\eqref{Hsigro} with \eqref{eqn:choice-c}, we can estimate for all $P \in \overline{Q(c)}$,
\[
\sum_{\ell = i + 1}^\infty\bar S_{i,\ell}(\bar\sigma_{i,\ell}\bar H_{\ell} +\bar H^1_\ell \bar \rho^1_{i,\ell} +\bar H^2_\ell \bar\rho^2_{i,\ell})(P) \ge (L_1m - L_2(c^{\min\{1,2(p-2)\}}+c^{\min\{1,p-2\}}))(i + 1)^{2(p-2)},
\]
for some constants $L_1(p), L_2(p) > 0$. If $c$ and $\eps$ of \eqref{tbsimp1} are chosen sufficiently small, the latter implies \eqref{concc} and hence \eqref{conc}.
\\
\\
\fbox{The case $1 < p < 2$:} we wish to show that, for all $c$ sufficiently large there exists a constant $C' = C'(c)> 0$, such that for all $P \in \overline{Q(c)}$ and for all $i \ge I = I(c)$:
\begin{equation}\label{concclessa}
\sum_{\ell = i + 1}^\infty a\bar S_{i,\ell}( \bar\sigma_{i,\ell}\bar H_{\ell} + \bar H^1_\ell \bar \rho^1_{i,\ell} +\bar H^2_\ell \bar\rho^2_{i,\ell})(P) \ge C'i^{2(p-2)}.
\end{equation}
This would show, since $a \in (c,2c)$, that
\[
 \sum_{\ell = i + 1}^\infty\bar S_{i,\ell}( \bar\sigma_{i,\ell}\bar H_{\ell} + \bar H^1_\ell \bar \rho^1_{i,\ell} +\bar H^2_\ell \bar\rho^2_{i,\ell})(P) \ge \frac{C'}{2c}i^{2(p-2)}.
\]
As above, by \eqref{eqn:choice-c}-\eqref{tbsimp1} and \eqref{tail}, the latter would imply \eqref{conc}, once $\eps$ of \eqref{tbsimp1} is chosen small enough in terms of $\frac{C'}{2c}$. To show \eqref{concclessa}, we introduce the following notation: here $g_k(a)$, for $k \in \N$, is a continuous function $g_k: \R^+ \to \R$ such that for all $a$ sufficiently large, one has
\[
|g_k(a)| \le \frac{L_k}{a^{p-1}},
\]
for some constant $L_k > 0$. With this notation, we rewrite for all $a \in \R^+$, $\ell \ge i + 1\ge 1$,
\[
a\bar \sigma_{i,\ell} =   \frac{a^2}{(1 + a)^2}\left[\frac{i^2}{\ell^2}-\frac{i^2}{(i+1)^2}\right] + g_1(a)\ln\left(\frac{\ell}{i +1}\right) + g_2(a)\left[\frac{i^2}{\ell^2} - \frac{i^2}{(i+1)^2}\right],
\]
and
\[
\bar H_\ell = \frac{2aw}{ a + 1}\ell^{2(p-2)-1} + g_3(a) \ell^{2(p-2) - 1}.
\]
Therefore,
\begin{equation}\label{a:first}
\begin{split}
a\bar \sigma_{i,\ell}\bar H_\ell &= \frac{2a^3w}{(1 + a)^3}\ell^{2(p-2) - 1}\left[\frac{i^2}{\ell^2}-\frac{i^2}{(i+1)^2}\right] \\
&\quad+ \left(g_5(a)\ln\left(\frac{\ell}{i +1}\right) + g_6(a)\left[\frac{i^2}{\ell^2} - \frac{i^2}{(i+1)^2}\right]\right) \ell^{2(p-2)-1}.
\end{split}
\end{equation}
The term $a( \bar H^1_\ell \bar \rho^1_{i,\ell} +\bar H^2_\ell \bar\rho^2_{i,\ell})$ is more complicated, but it admits the similar representation:
\begin{equation}\label{a:second}
\begin{split}
a( \bar H^1_\ell \bar \rho^1_{i,\ell} +\bar H^2_\ell \bar\rho^2_{i,\ell}) &= \frac{2aw}{a + 1}\ell^{2(p-2)-1} - \frac{2a^2w}{(a+1)^2}\ell^{2(p-2)-1}\frac{(\ell - 1)^2 - i^2}{\ell^2}\\
&\quad+ g_6(a)\ell^{2(p-2)-2} + (g_7(a) + \ell g_8(a))\ell^{2(p-2)-2}\frac{(\ell - 1)^2 - i^2}{\ell^2}.
\end{split}
\end{equation}
Now we can exploit \eqref{Euleralpha}-\eqref{a:first}-\eqref{a:second}-\eqref{tail1}-\eqref{tail2}, and the definition of $\bar S_{i,\ell}$ of \eqref{def:Sbar} to write
\begin{align*}
\sum_{\ell = i + 1}^\infty &a\bar S_{i,\ell}\bar\sigma_{i,\ell}\bar H_{\ell}(P) =  f_1(a)w(i + 1)^{2(p-2)} + R_{1,i}(a,w) \\
& =\left[\frac{1}{2(2-p) + 2 + 2G_p(a)}-\frac{1}{2(2-p) + 2G_p(a)}\right]\frac{2a^3w}{(1 + a)^3}(i+1)^{2(p-2)} + R_{1,i}(a,w),
\end{align*}
and
\begin{align*}
&\sum_{\ell = i + 1}^\infty a\bar S_{i,\ell}( \bar H^1_\ell \bar \rho^1_{i,\ell} +\bar H^2_\ell \bar\rho^2_{i,\ell})(P) = f_2(a)w(i + 1)^{2(p-2)}+ f_3(a)w(i + 1)^{2(p-2)} + R_{2,i}(a,w)\\
&\quad = \left[\frac{1}{2(2-p) + 2G_p(a)}\right]\frac{2aw(i+1)^{2(p-2)}}{a + 1}\\
&\qquad\quad-\left[\frac{1}{2(2-p)+2G_p(a)}-\frac{1}{2(2-p)+2+2G_p(a)}\right]\frac{2a^2w(i+1)^{2(p-2)}}{(a+1)^2} + R_{2,i}(a,w).
\end{align*}
Here, $R_{j,i}(a,w)$ are continuous functions for all $j = 1,2$, $i \in \N$, with the property that
\[
|R_{1,i}(a,w)| + |R_{2,i}(a,w)| \le \frac{k(i + 1)^{2(p-2)}}{a^{p-1}},
\]
for all $i \in N$, $w \in \left(\frac{3}{4},\frac{5}{4}\right)$ and $a > 0$ sufficiently large, and for a positive constant $k$ depending solely on $p$ once $c$ is chosen large enough. Notice that $R_{3,i}(a,w) \doteq R_{a,i}(a,w) + R_{2,i}(a,w)$ enjoys the same bounds. As $w \in \left(\frac{3}{4},\frac{5}{4}\right)$, in order to show \eqref{concclessa}, it is sufficient to prove that
\[
\liminf_{a\to \infty}f_{1}(a) + f_2(a) + f_3(a) > 0.
\]
A direct computation that exploits $\lim_{a\to \infty}G_p(a) = 1$, see \eqref{defG}, shows that
\[
\lim_{a \to \infty}f_{1}(a) + f_2(a) + f_3(a) = \frac{2(2-p)}{(2(2-p)+4)(2(2-p) + 2)},
\]
which is positive for all $p < 2$. This concludes \eqref{concclessa}.
\\
\\
We collect the results we just obtained in the following:

\begin{prop}\label{ansmag2}
Let $p \neq 2$, $p \in(1,+\infty)$ and $S_{i}$, $H_\ell$, $\rho^j_{i,\ell}$, $\sigma_{i,\ell}$ be defined in \eqref{Sell}-\eqref{Hell}-\eqref{deltaHell1}-\eqref{deltaHell2}-\eqref{sigmaiell} respectively. Recall that $Q(c) = (c,2c)\times \left(\frac{3}{4},\frac{5}{4}\right)\times \left(\frac{3}{4},\frac{5}{4}\right)\times \left(\frac{3}{4},\frac{5}{4}\right)$. Then, for all $c > 0$ sufficiently small if $p > 2$ and $c > 0$ sufficiently large if $p < 2$, there exists $I = I(c) \in \N$ such that for all $i \ge I$:
\begin{enumerate}[(a)]
\item\label{a1} The convergence of the series $\sum_{\ell = 1}^\infty \frac{H_\ell(P)}{S_\ell(P)}$ is uniform in $\overline{Q(c)}$. In particular, 
\[
P \mapsto z_i(P) = S_i(P)\sum_{\ell = i + 1}^\infty \frac{H_\ell(P)}{S_\ell(P)}
\]
is well-defined for all $P \in \overline{Q(c)}$ and continuous on $\overline{Q(c)}$;
\item\label{b1} $\sup_{P \in \overline{Q(c)}}|z_i(P)| \le C i^{2(p-2)}$, for all $i \ge 1$, and for some positive $C > 0$;
\item\label{c1} The convergence of the series $\sum_{\ell = 1}^\infty \frac{1}{S_\ell}[H_\ell\sigma_{i,\ell} + +H^1_\ell \rho^1_{i,\ell} +H^2_\ell \rho^2_{i,\ell}]$ is uniform in $Q(c)$. In particular,  \[
(\partial_a z_i - i^{2} \partial_{x_0}z_i)(P)
\]
exists at all points $P \in Q(c)$, is continuous, and is represented precisely by
\[
(\partial_a z_i - i^{2} \partial_{x_0}z_i)(P) =\sum_{\ell = 1}^\infty \frac{S_i}{S_\ell}[-H_\ell\sigma_{i,\ell} +H^1_\ell \rho^1_{i,\ell} +H^2_\ell \rho^2_{i,\ell}].
\]
\item\label{d1} at all $P \in Q(c)$,
\[
(\partial_a z_i - i^{2} \partial_{x_0}z_i)(P) > 0.
\]
In particular, the assumptions of Lemma \ref{simpl} hold true and the mappings $\Phi_{i,t}^1,\Phi_{i,t}^2,A_i$ are open for all $t \in[0,1)$ and $i \ge I$.
\end{enumerate}
\end{prop}
\begin{proof}
We see from \eqref{estH} that
\begin{equation}\label{estglob}
\max_{P \in \overline{Q(c)}}|H_\ell(P)| \le C\ell^{2(p-2) - 1}.
\end{equation}
Having chosen $c$ such that \eqref{eqn:choice-c} hold, we estimate for all $P = (a,x_0,y_0,w) \in \overline{Q(c)}$
\[
\left|\frac{S_i(P)}{S_\ell(P)}H_\ell(P)\right| \overset{\eqref{SIL}-\eqref{estglob}}{\le}C\left(\frac{i + 1}{\ell}\right)^{2 G_p(a)}\ell^{2(p-2)-1} \overset{\eqref{eqn:choice-c}}{\le} C\left(\frac{i + 1}{\ell}\right)^{2 \max\{1,p-1\}}\ell^{2(p-2)-1},
\]
for all $\ell \ge i + 1$. Through \eqref{ziexp} and \eqref{Euleralpha}, we conclude that $z_i$ is well-defined and continuous for all $i \ge 0$, that is \eqref{a1}, and that it enjoys property $\eqref{b1}$. Finally, \eqref{c1} readily follows from \eqref{aboveest}-\eqref{tbsimp1}, \eqref{SIL} and \eqref{Euleralpha}-\eqref{tail}, and \eqref{d1} is the content of \eqref{conc}.
\end{proof}

In what follows, $c$ is chosen so that \eqref{incre}-\eqref{increless} and Proposition \ref{ansmag2} hold, and will be considered a fixed parameter.

\section{Some further properties of $\Phi^k_{i,t}$ and $A_i$}\label{further}

As in the previous section, the domain of the maps we will consider is $Q(c)$, where $c>0$ is fixed by \eqref{incre}-\eqref{increless} and Proposition \ref{ansmag2}. We also let $I = I(c)$ the index for which the conclusion of the aforementioned proposition holds. This index will be made (possibly) larger in the next Lemma, but will still be denoted by $I$. Denote, as before, points of $Q(c)$ by $P$.

\begin{lemma}\label{sigma}
There exists $0 < t_0 = t_0(c)< 1$ such that if $t \in [t_0,1]$, the sets
\[
\Phi^1_{i_1,t}(\overline{Q(c)}),\quad \Phi^1_{i_2,t}(\overline{Q(c)}) \text{ and } A_{i_3}(\overline{Q(c)})
\]
are pairwise disjoint for all $i_1,i_2,i_3 \ge I$, where $I = I(c)$ may be larger than the one of Propositions \ref{ansmag2}.
\end{lemma}
\begin{proof}
We first show the following auxiliary statements. There exist $k > 0$, $I \in \N$ and $0 < t_0 < 1$ such that $\forall i \ge I, t \in [t_0,1], P = (a,x_0,y_0,w) \in \overline{Q(c)}$
\begin{equation}\label{primo}
\Phi^1_{i,t}(x_0,y_0,z_0,w) \subset \left\{\left(\begin{array}{cc} x & w\\ z& y\end{array}\right) \in \R^{2\times 2}: x \ge k, y \le -k\right\}, 
\end{equation}
\begin{equation}\label{secondo}
\Phi^2_{i,t}(x_0,y_0,z_0,w) \subset \left\{\left(\begin{array}{cc} x & w\\ z& y\end{array}\right) \in \R^{2\times 2}: x \le -k, y \ge k\right\},
\end{equation}
\begin{equation}\label{terzo}
A_i(a,x_0,y_0,w) \subset  \left\{\left(\begin{array}{cc} x & w\\ z& y\end{array}\right) \in \R^{2\times 2}: x \ge k, y \ge k\right\}.
\end{equation}
\\
\\
\fbox{Proof of \eqref{primo}.} To prove \eqref{primo}, we only need to show that $y_{i-1} - t(y_{i-1} + g_{w}(x_{i-1}))$ is (uniformly) negative for all $P \in \overline{Q(c)}$ if $t_0$ is sufficiently close to $1$ and $i \ge I$. Indeed, the fact that $P\mapsto x_{i-1}(P)$ is (uniformly) positive is an immediate consequence of its definition. We have
\begin{align*}
y_{i-1} -t(&y_{i-1} + g_w(x_{i-1})) = (1-t)((i-1)^{2(p-1)} + y_0) - t|(a(i-1)^2 + x_0,w)|^{p-2}(a(i-1)^2 + x_0).
\end{align*}
If we divide the latter by $i^{2(p-1)}$ and send $i \to \infty$, we obtain, uniformly in $P \in \overline{Q(c)}$,
\[
(1-t) - a^{p-1}t,
\]
and we can estimate
\[
(1-t) - a^{p-1}t  \le (1-t) - c^{p-1}t 
\]
Therefore, if $t_0$ is sufficiently close to $1$, we see that \eqref{primo} holds for $I$ large.
\\
\\
\fbox{Proof of \eqref{secondo}.} Analogously to the proof of \eqref{primo}, we have that $y_i$ is always positive, and hence we shall only prove that $x_{i-1} - t(g^{-1}_w(y_i) + x_{i-1})$ is (uniformly) negative for all $P \in \overline{Q(c)}$ if $t_0$ is sufficiently close to $1$ and $i \ge I$. To do so, we write:
\begin{align*}
x_{i-1} - t(g^{-1}_w(y_i) - x_{i-1}) &= (1-t)x_{i-1} - tg^{-1}_w(y_i).
\end{align*}
Recalling \eqref{invg}, we divide the latter by $i^2$ and we let $i \to \infty$ to obtain, uniformly in $P \in \overline{Q(c)}$,
\[
(1-t)a - t \le (1-t)2c - t.
\]
Therefore, if $t_0$ is sufficiently close to $1$, \eqref{secondo} holds for $I$ large and uniformly in $P \in \overline{Q(c)}$.
\\
\\
\fbox{Proof of \eqref{terzo}.} This is immediate, since $\{x_i\}, \{y_i\}$ are uniformly positive by their definition.
\\
\\
Clearly, \eqref{primo}-\eqref{secondo}-\eqref{terzo} imply that
\[
\Phi^{1}_{i,t}(\overline{Q(c)})\cap \Phi^{2}_{j,t}(\overline{Q(c)}) = \emptyset,\quad \Phi^{1}_{i,t}(\overline{Q(c)})\cap A_j(\overline{Q(c)}) = \emptyset, \quad \Phi^{2}_{i,t}(\overline{Q(c)})\cap A_j(\overline{Q(c)}) = \emptyset,
\]
provided $i, j \ge I$, $t \in [t_0,1]$, $I$ is large enough and $t_0$ is close enough to $1$. To conclude the proof of the Lemma, we now claim that if $I$ is sufficiently large, then for all $i > j > I$:
\begin{equation}\label{claimm}
x_i(x_0) > x_{j}(x_0') + 1, \quad y_i(y_0) > y_{j}(y_0') + 1, \quad \forall i > j \ge I, \quad \forall x_0,y_0,x_0',y_0' \in \left(\frac{3}{4},\frac{5}{4}\right).
\end{equation}
Using the definitions of $\Phi^{k}_{i,t}$, it is easy to see that \eqref{claimm} implies, $\forall k=1,2, i > j \ge I, t \in [t_0,1]$,
\begin{equation}\label{dist}
\dist(\Phi^k_{i,t}(\overline{Q(c)}),\Phi^k_{j,t}(\overline{Q(c)})) \ge 1,\quad \dist(A_i(\overline{Q(c)}),A_j(\overline{Q(c)})) \ge 1,
\end{equation}
that in turn yields the conclusion of the present proof.
\\
\\
We only need to show \eqref{claimm}. We can estimate
\[
ai^2- aj^2 \ge 2 aj(i-j) \ge 2 aj \ge 2 a I,\quad \forall i > j \ge I.
\]
Thus, recalling that $x_0,x_0' \in \left(\frac{3}{4},\frac{5}{4}\right)$,
\[
x_i(x_0)-x_j(x_0') = a(i^2 -j^{2}) + x_0 - x_0' \ge 2 a I - \frac{1}{2} \ge 2 c I - \frac{1}{2}.
\]
If $I = I(c) > 0$ is sufficiently large, we see that the first part of \eqref{claimm} holds. To show the second part, i.e. the analogous estimate for $y_i$, we can repeat the exact same proof with small modifications. This concludes the proof of the Lemma.
\end{proof}
We end this section with one last technical lemma. First, let us introduce the following sets. Define, for $k =1,2, i\ge 1$, $q \ge 1$ and $t_q \doteq 1 -\frac{1-t_0}{2^q}$,
\[
U^k_{i,q}\doteq \Phi_{i,t_q}^k(Q(c)) \quad\text{ and }\quad U_i^3 \doteq A_i(Q(c)).
\]
The reason to update $t_q \to 1$ is that, as $q\to \infty$, we wish the sets $U^k_{i,q}$ to converge to subsets of $K_p$ in a sense that is made rigorous in the next lemma.
Recall that $I \in \N$ is a sufficiently large index so that Proposition \ref{ansmag2} holds together with Lemma \ref{sigma}. Finally, set for $n \ge 1$
\begin{equation}\label{VN}
V_{n} \doteq \bigcup_{i = I}^{I + n-1} U_{i,I + n - 1}^1\cup \bigcup_{i = I}^{I + n-1} U_{i,I +n -1}^2 \cup U_{I + n}^3.
\end{equation}
Notice that by Lemma \ref{simpl} and Proposition \ref{ansmag2}, each of the above sets is open and, by Lemma \ref{sigma}, the union defining $V_n$ is disjoint, for all $n \ge I +1$. We will still need to update the value of $I$ in some of the next results. We modify the definition of $V_n$ accordingly.

\begin{lemma}\label{conv}
Let $Z_n \in V_n$, $\forall n \ge 1$. Suppose that a subsequence $\{Z_{n_j}\}_j$ converges to $Z$. Then,
\begin{equation}\label{ZZ}
Z \in \bigcup_{i=I}^\infty \Phi^1_{i,1}(\overline{Q(c)})\cup\bigcup_{i= I}^\infty\Phi^2_{i,1}(\overline{Q(c)}) \overset{\eqref{endpoints}}{=} \bigcup_{i=I}^\infty B_i(Q(c))\cup\bigcup_{i=I}^\infty D_i(Q(c)) \overset{\eqref{BE}-\eqref{CD}}{\subset} K_{p}.
\end{equation}
\end{lemma}
\begin{proof}
Denote by $B_R(0) \subset \R^{2\times 2}$ the ball of radius $R$ centered at $0 \in \R^{2\times 2}$. Since 
\[
\inf_{P \in \overline{Q(c)}}x_i(P) \to \infty \quad\text{ and }\quad \inf_{P \in \overline{Q(c)}}y_i(P) \to \infty
\]
as $i \to \infty$, we see that given any $R > 0$:
\begin{equation}\label{PHI}
A_i(Q(c))\cap B_R(0) = \emptyset, \quad \Phi^{k}_{i,t}(Q(c))\cap B_R(0) = \emptyset,
\end{equation}
for all $t \ge t_0, k =1,2, i \ge \bar I = \bar I(R) \ge 1$. Since $\{Z_{n_j}\}_j$ is convergent, it is bounded, and from $Z_{n_j} \in V_{n_j}, \forall j$, and \eqref{PHI}, we infer
\[
Z_{n_j} \in \bigcup_{i=I}^{N} U_{i,n_j}^1\cup\bigcup_{i=I}^{N} U_{i,n_j}^2 \subset V_{n_j}, \quad \forall j \ge J,
\]
for some large $J,N$, the latter independent of $j$. Now, either for $k=1$ or $k =2$, we can select a (non-relabeled) subsequence of $\{Z_{n_j}\}$ with the property that, for some $i_0 \in \{I,\dots, N\}$,
\[
Z_{n_j} \in U_{i_0,n_j}^k.
\]
Let $k =1$ for simplicity, otherwise the proof is analogous. By definition of $U_{i_0,n_j}^k$, we find a sequence $P_j = (a_j,x_0^j,y_0^j,w_j) \in Q(c)$ such that, for all $j \in \N$,
\[
Z_{n_j} = \Phi_{i_0,t_{n_j}}^1(P_j).
\]
Up to passing to a subsequence, there exists a limit $\lim_j P_j = P \in \overline{Q(c)}$. Now observe that the map $(t,a,x_0,y_0,w) \in ([0,1]\times\overline Q(c)) \mapsto \Phi_{i_0,t}^k(x_0,y_0,z_0,w)$ is continuous for $k=1,2$ on $\overline{Q(c)}$. This can be inferred from the definitions \eqref{phi1}-\eqref{phi2}, the continuity of $z_i$, proved in Propositions \ref{ansmag2}, and \eqref{vwi}. Continuity tells us that $Z = \Phi_{i_0,1}^1(P)$ and concludes the proof.
\end{proof}

\section{Properties of the laminates}\label{sec:proplam}
Aim of this section is to build the laminates that will be used in the proof of the inductive proposition of Section \ref{sec:ind}. In what follows, $c$ is a fixed parameter, that has been chosen in \eqref{incre}-\eqref{increless} and Proposition \ref{ansmag2}. Moreover,
\[
t_q = 1 -\frac{1-t_0}{2^q},
\]
for all $q \in \N$, for $t_0$ of Lemma \ref{sigma}. Finally, we denote by $I = I(c) \in \N$ the index for which the Proposition \ref{ansmag2} and Lemma \ref{sigma} hold, for all $i \ge I$. We will update this index a few more times in this section, and this will fix $I$ for the last section. We will denote by $[X,Y]$ the segment connecting $X,Y \in \R^{2\times 2}$, and we will call it a rank-one segment if $\det(X-Y) = 0$.

\begin{lemma}\label{costrone1}
Let $i \ge 1$, $q \ge 1$. For all $M \in Q(c)$, the matrix $\Phi_{i,t_q}^1(M)$ belongs to the rank-one segment $\left[\Phi_{i,t_{q+1}}^1(M),A_{i}(M)\right]$:
\[
\Phi_{i,t_q}^1(M) = \mu_{1,q}\Phi_{i,t_{q+1}}^1(M) + \mu_{2,q}A_{i}(M)
\]
with
\[
\mu_{1,q} =\frac{t_q}{t_{q + 1}}  \text{ and } \mu_{2,q} = 1-\frac{t_q}{t_{q + 1}}.
\]
\end{lemma}

\begin{proof}
By the definitions, we immediately see that $\det(\Phi_{i,t_{q+1}}^1(M) - A_i(M)) = 0$. Therefore the segment $\left[\Phi_{i,t_{q+1}}^1(M),A_{i}(M)\right]$ is a rank-one segment. Using \eqref{phi1imp}, we compute
\begin{align*}
\mu_{1,q}\Phi^1_{i,t_{q+1}}(M) &+ \mu_{2,q}A_i(M) \\
&= \frac{t_q}{t_{q+1}}(A_i(M) + t_{q+1}\lambda_{E_i(M)}(B_i(M) - E_i(M))) + \frac{t_{q+1}-t_{q}}{t_{q + 1}} A_i(M)\\
& = A_i(M) + t_{q}\lambda_{E_i(M)}(B_i(M) - E_i(M)) = \Phi^1_{i,t_q}(M).
\end{align*}
\end{proof}

\begin{lemma}\label{costrone3}
Let $i \ge 1$, $q \ge 1$. For all $M \in Q(c)$, the matrix $\Phi_{i,t_q}^2(M)$ belongs to the rank-one segment $\left[\Phi_{i,t_{q+1}}^2(M),C_{i}(M)\right]$:
\begin{equation}\label{eq:eq}
\Phi_{i,t_{q}}^2(M) = \mu_{i,q}^1(M)\Phi_{i,t_{q + 1}}^2(M) +  \mu_{i,q}^2(M)C_{i}(M)
\end{equation}
with
\[
\mu_{i,q}^1(M) =\frac{\lambda_{D_i(M)} +t_q\lambda_{C_i(M)}}{\lambda_{D_i(M)} +t_{q+1}\lambda_{C_i(M)}}  \text{ and } \mu_{i,q}^2(M) = \frac{(t_{q+ 1}-t_q)\lambda_{C_i(M)}}{\lambda_{D_i(M)} +t_{q+1}\lambda_{C_i(M)}}.
\]
Furthermore, if $I$ is sufficiently large, there exists a dimensional constant $C> 0$ such that for all $q \ge I$, $i \ge 1$ and $M \in Q(c)$,
\begin{equation}\label{estcostrone}
1 - \frac{C}{2^q} \le \mu_{i,q}^1(M)<1 \text{ and }  0 < \mu_{i,q}^2(M) \le \frac{C}{2^q}.
\end{equation}
\end{lemma}
\begin{proof}
All the assertions can be checked by direct computation. First, using the definitions, we see that $\det(\Phi^2_{i,t_{q + 1}}(M) - C_{i}(M)) = 0$. Therefore, the segment $\left[\Phi_{i,t_{q+1}}^2(M),C_{i}(M)\right]$ is a rank-one segment. Moreover, using \eqref{phi2new}, one directly checks \eqref{eq:eq}. We only need to show \eqref{estcostrone}. The first estimate follows from the second and the fact that $\mu_{i,q}^1 + \mu_{i,q}^2 = 1$. Recall that $t_q = 1 -\frac{1-t_0}{2^q} < 1$. To show the second estimate of \eqref{estcostrone}, we simply write
\[
0 < \mu_{i,q}^2(M) = (t_{q+ 1}-t_q)\frac{\lambda_{C_i(M)}}{\lambda_{D_i(M)} +t_{q+1}\lambda_{C_i(M)}} \le  \frac{t_{q+ 1}-t_q}{t_{q+1}} = \frac{1-t_0}{t_{q + 1}2^{q + 1}}.
\]
Since $\lim_{q \to \infty}t_q = 1$, we conclude the validity of \eqref{estcostrone}.
\end{proof}

\begin{prop}\label{costrlam}
Let $q \ge j_2-2, j_2 > j_1 \ge 2 $. Let moreover $P \in Q(c)$. Then, $A_{j_1}(P)$ is the barycenter of a laminate of finite order $\nu_{j_1,j_2,q}(P)$ with 
\begin{equation}\label{supp}
\spt(\nu_{j_1,j_2,q}(P)) \subset \bigcup_{k = j_1}^{j_2 -1}U_{k,q}^1\cup \bigcup_{k = j_1}^{j_2-1}U_{k,q}^2\cup U^3_{j_2}.
\end{equation}
Furthermore, 
\begin{equation}\label{lam:positive}
0 <\nu_{j_1,j_2,q}(P)(U^3_{j_2})
\end{equation}
\end{prop}

The proof of the previous proposition is inductive. We record the base case separately in the following

\begin{lemma}\label{base}
Let $i \ge I + 1$, $q \ge I$, and $P \in Q(c)$. Then, $A_{i}(P)$ is the barycenter of the laminate of finite order
\[
\nu_{i,q}(P) \doteq \lambda_{i,q}^1(P)\delta_{\Phi^1_{i,t_q}(P)} + \lambda_{i,q}^2(P)\delta_{\Phi^2_{i,t_q}(P)} + \lambda_{i,q}^3(P)\delta_{A_{i + 1}(P)},
\]
where
\begin{align}
\label{w1}\lambda_{i,q}^1(P) &= \frac{\lambda_{B_i(P)}}{\lambda_{B_i(P)} + t_q\lambda_{E_i(P)}},\\
\label{w2}\lambda_{i,q}^2(P) &= \frac{t_q\lambda_{E_i(P)}}{\lambda_{B_i(P)} + t_q\lambda_{E_i(P)}}\frac{\lambda_{D_i(P)}}{t_q\lambda_{C_i(P)} + \lambda_{D_i(P)}},\\
\label{w3}\lambda_{i,q}^3(P) &= \frac{t_q\lambda_{E_i(P)}}{\lambda_{B_i(P)} + t_q\lambda_{E_i(P)}}\frac{t_q\lambda_{C_i(P)}}{t_q\lambda_{C_i(P)} + \lambda_{D_i(P)}}.
\end{align}
Furthermore, independently of $P$, if $I$ is large enough, there exist positive constants $C$, $k_1 < k_2$ such that
\begin{align}
\label{w1def>2}\frac{k_1}{i} &\le\lambda_{i,q}^1(P) \le \frac{k_2}{i},\\
\label{w2def>2}\frac{k_1}{i} &\le \lambda_{i,q}^2(P)  \le \frac{k_2}{i},\\
\label{w3def>2}0&< \lambda_{i,q}^3(P)  \le e^{\frac{C}{i^\gamma}}e^{-2\frac{\min\{G_p(c),G_p(2c)\}}{i}},
\end{align}
where $\gamma = 1 + \min\{1,2(p-1)\}$.
\end{lemma}
\begin{proof}
From $\delta_{A_i(P)}$ we pass to $\nu_{i,q}(P)$ via two consequent elementary splittings, in the sense of Definition \ref{sec:def}. These splitting were already mentioned when introducing the maps under consideration, see Remark \ref{lof:BDC}. First, we split $A_i(P)$ in a rank-one segment with direction $B_i(P) - E_i(P)$. The segment has endpoints $E_i(P)$ and $\Phi^1_{i,t_q}(P)$. Then, we split again $\delta_{E_i(P)}$ into a rank-one segment with direction $C_i(P)-D_i(P)$, with endpoints $C_i(P) = A_{i + 1}(P)$ and $\Phi^2_{i,t_q}(P)$. Weights \eqref{w1}-\eqref{w2}-\eqref{w3} are obtained via direct computation. We now turn to the proof of \eqref{w1def>2}-\eqref{w2def>2}-\eqref{w3def>2}. First, we notice that, as $t_q = 1 - \frac{1-t_0}{2^q}$, $t_q \to 1$ as $q \to \infty$. Therefore, if $I$ is sufficiently large,
\[
\frac{1}{2} \le t_q \le 1.
\]
Furthermore, as $\lambda_{E_i(P)}+\lambda_{B_i(P)}=1$, it follows that
\begin{equation}\label{convsum1}
1 \ge \lambda_{B_i(P)} + t_q\lambda_{E_i(P)} \ge t_q \ge \frac{1}{2}.
\end{equation}
Analogously, we have
\begin{equation}\label{convsum2}
1 \ge \lambda_{D_i(P)} + t_q\lambda_{C_i(P)} \ge t_q \ge \frac{1}{2}.
\end{equation}
Combining these estimates, we find the first bounds
\begin{align*}
&k_1\lambda_{B_i(P)} \le \lambda_{i,q}^1(P) \le k_2\lambda_{B_i(P)},\\
&k_1\lambda_{E_i(P)}\lambda_{D_i(P)} \le \lambda_{i,q}^2(P) \le k_2\lambda_{E_i(P)}\lambda_{D_i(P)}.
\end{align*}
Using the notation of Section \ref{sec:def}, it is easy to see that
\[
\lambda_{B_i(P)} \sim \frac{2(p-1)}{1 + a^{p-1}}\frac{1}{i},\quad \lambda_{E_i(P)}\lambda_{D_i(P)} \sim \frac{a}{a +1}\frac{1}{i},
\]
whence \eqref{w1def>2}-\eqref{w2def>2} follow.  We now move to the proof of \eqref{w3def>2}. The bound from below is immediate. Moreover, by \eqref{convsum1}-\eqref{convsum2}, we find that
\[
0 < \lambda_{i,q}^3(P) \le \lambda_{E_i(P)}\lambda_{C_i(P)} = \frac{y_{i-1} + g_{w}(x_{i-1})}{y_i + g_{w}(x_{i-1})}\frac{x_{i-1}+g^{-1}_w(y_i)}{x_{i}+g^{-1}_w(y_i)}.
\]
We need to show that, for $\gamma = 1 + \min\{1,2(p-1)\}$, some $C > 0$ and large $i$,
\[
\frac{y_{i-1} + g_{w}(x_{i-1})}{y_i + g_{w}(x_{i-1})}\frac{x_{i-1}+g^{-1}_w(y_i)}{x_{i}+g^{-1}_w(y_i)} \le e^\frac{C}{i^\gamma}e^{-2 \min\{G_p(c),G_p(2c)\}\frac{1}{i}}.
\]
This was already proved in Section \ref{ASYSIL}, and we sketch the estimate for the convenience of the reader. We only treat the estimate of $\lambda_{C_i(P)}$, the proof of the other term being analogous. Write
\[
\frac{x_{i-1}+g^{-1}_w(y_i)}{x_{i}+g^{-1}_w(y_i)} = e^{\ln(x_{i-1}+g^{-1}_w(y_i))- \ln(x_{i}+g^{-1}_w(y_i)}.
\]
Now use \eqref{log1} to estimate
\[
\ln(x_{i-1}+g^{-1}_w(y_i))- \ln(x_{i}+g^{-1}_w(y_i)) \le \frac{x_{i-1} - x_i}{x_{i} + g_w^{-1}(y_i)}.
\]
Estimate \eqref{eqn:claim-main} 
 gives us
\[
e^{\ln(x_{i-1}+g^{-1}_w(y_i))- \ln(x_{i}+g^{-1}_w(y_i))} \le e^\frac{C}{i^\gamma}e^{-\frac{2 a}{a +1}\frac{1}{i}}.
\]
A similar proof yields
\[
 \frac{y_{i-1} + g_{w}(x_{i-1})}{y_i + g_{w}(x_{i-1})} \le e^\frac{C}{i^\gamma}e^{-\frac{2(p-1) }{1 + a^{p-1}}\frac{1}{i}}.
\]
Combining the last two inequalities, we find, for all $P \in \overline{Q(c)}$
\[
0 < \lambda_{i,q}^3(P) \le e^\frac{C}{i^\gamma}e^{-2 G_p(a)\frac{1}{i}} \overset{\eqref{incre}-\eqref{increless}}{\le} e^\frac{C}{i^\gamma}e^{-2 \min\{G_p(c),G_p(2c)\}\frac{1}{i}}
\]
\end{proof}

We can now prove Proposition \ref{costrlam}.

\begin{proof}[Proof of Proposition \ref{costrlam}.]
The construction is inductive. If $j_2 = j_1 + 1$, then we set 
\[
\nu_{j_1,j_2,q}(P) \doteq \nu_{j_1,q}(P),
\]
where $\nu_{j_1,q}$ was introduced in Lemma \ref{base}. If $j_2 > j_1 + 1$, then $\nu_{j_1,j_2,q}(P)$ is obtained via $j_2 - j_1 \ge 2$ steps. First, $\mu^{1} \doteq \nu_{j_1,q}(P)$. By definition, $\mu^{1}$ contains a Dirac's delta at $C_{j_1}(P) = A_{j_1 + 1}(P)$ with weight $\lambda^3_{j_1,q}(P)$. Now $A_{j_1 + 1}(P)$ is again the barycenter of $\nu_{j_1 + 1,q}(P)$ of Lemma \ref{base}. Hence we set
\[
\mu^2 \doteq \mu^1 - \lambda^3_{j_1,q}\delta_{A_{j_1 + 1}(P)} + \lambda^3_{j_1,q}\nu_{j_1 + 1,q}(P).
\]
If $j_2 - j_1 = 2$, we stop. Otherwise, we continue iteratively, defining for $k \in \{2,\dots, j_2 - j_1\}$
\begin{equation}\label{def:muk}
\mu^k \doteq \mu^{k - 1} - \left(\prod_{r = 0}^{k-2}\lambda^3_{j_1+ r,q}(P)\right)\delta_{A_{j_1 + k - 1}(P)}+ \left(\prod_{r = 0}^{k -2}\lambda^3_{j_1+ r,q}(P)\right)\nu_{j_1 +k - 1,q}(P).
\end{equation}
Finally, the required laminate is $\nu_{j_1,j_2,q}(P) \doteq \mu^{j_2 - j_1}$. By construction, $\nu_{j_1,j_2,q}(P)$ is a laminate of finite order. The fact that the barycenter of $\nu_{j_1,j_2,q}(P)$ is $A_{j_1}(P)$ and \eqref{supp} also follow by construction. By Lemma \ref{sigma}, all the sets $U_{k,q}^1,U_{k,q}^2$ and  $U^3_{j_2}$, for $k \in \{j_1,\dots,j_2-1\}$ are pairwise disjoint. Therefore,
\[
\nu_{j_1,j_2,q}(P)(U^3_{j_2}) =  \left(\prod_{r = 0}^{j_2-j_1 -2}\lambda^3_{j_1+ r,q}(P)\right)\nu_{j_2 - 1,q}(P)(U^3_{j_2}) \overset{\eqref{w3def>2}}{>} 0
\]
\end{proof}

We combine Lemmas \ref{costrone1}-\ref{costrone3} and Proposition \ref{costrlam} to prove the last two Propositions of this section.

\begin{prop}\label{comb1}
Let $q \ge j_2-2, j_2 > j_1$. There exists a sufficiently large $I \in \N$ such that if $j_1 \ge I$, then, there exists a universal constant $C > 0$ such that the following hold. For all $P \in Q(c)$, $\Phi^1_{j_1,t_{q}}(P)$ is the barycenter of a laminate of finite order $\nu^1_{j_1,j_2,q}(P)$ with
\[
\spt(\nu^1_{j_1,j_2,q}(P)) \subset \bigcup_{i = j_1}^{j_2-1}U_{i,q + 1}^1\cup \bigcup_{i = j_1}^{j_2-1}U_{i,q + 1}^2 \cup U_{j_2}^3
\]
and the following estimates hold:
\begin{align}
\label{mixed11}\nu^1_{j_1,j_2,q}(P)(U^1_{j_1,q + 1}) &\ge 1 - \frac{C}{2^{j_2}},\\
\label{mixed21}\nu^1_{j_1,j_2,q}(P)\left( \bigcup_{i = j_1 + 1}^{j_2-1}U_{i,q + 1}^1\cup \bigcup_{i = j_1}^{j_2-1}U_{i,q + 1}^2 \cup U_{j_2}^3\right) &\le \frac{C}{2^{j_2}},\\
\label{mixed31}\nu^1_{j_1,j_2,q}(P)(U_{j_2}^3) &> 0.
\end{align}
\end{prop}
\begin{proof}
Combining Lemma \ref{costrone1} with Proposition \ref{costrlam}, we define the laminate of finite order
\[
\nu^1_{j_1,j_2,q}(P) = \frac{t_q}{t_{q + 1}}\delta_{\Phi^1_{j_1,t_{q + 1}}(P)} + \left(1 - \frac{t_q}{t_{q + 1}}\right)\nu_{j_1,j_2,q+1}(P).
\]
We check as in Lemma \ref{costrone1} that $\Phi^1_{j_1,t_{q}}(P)$ is the barycenter of $\nu^1_{j_1,j_2,q}(P)$. It is immediate to see that
\[
\spt(\nu^1_{j_1,j_2,q}(P)) \subset \bigcup_{i = j_1}^{j_2-1}U_{i,t_{q + 1}}^1\cup \bigcup_{i = j_1}^{j_2-1}U_{i,t_{q + 1}}^2 \cup U_{j_2}^3.
\]
Recall that the union is disjoint if $I$ is sufficiently large by Lemma \ref{sigma}. We now come to the required estimates. First,
\[
\nu^1_{j_1,j_2,q}(P)(U^1_{j_1,q + 1}) =  \frac{t_q}{t_{q + 1}} + \left(1 - \frac{t_q}{t_{q + 1}}\right)\nu_{j_1,j_2,q+1}(P)(U^1_{j_1,q + 1}) \ge \frac{t_q}{t_{q + 1}}.
\]
Since $q \ge j_2-2$ and $t_q = 1 - \frac{1-t_0}{2^q}$, we find a constant $C >0$ such that
\[
\frac{t_q}{t_{q + 1}} \ge 1 - \frac{C}{2^{j_2}}.
\]
It also follows that
\[
\nu^1_{j_1,j_2,q}(P)\left( \bigcup_{i = j_1 + 1}^{j_2-1}U_{j_1,q + 1}^1\cup \bigcup_{i = j_1}^{j_2-1}U_{j_1,q + 1}^2 \cup U_{j_2}^3\right) \le \frac{C}{2^{j_2}}.
\]
Finally, \eqref{mixed31} follows from \eqref{lam:positive}. This finishes the proof.
\end{proof}
\begin{prop}\label{comb2}
Let $q \ge j_2-2, j_2 > j_1+1$. There exists a sufficiently large $I \in \N$ such that if $j_1 \ge I$, then, there exists a universal constant $C > 0$ such that the following hold. For all $P \in Q(c)$, $\Phi^2_{j_1,t_{q}}(P)$ is the barycenter of a laminate of finite order $\nu^2_{j_1,j_2,q}(P)$ with
\[
\spt(\nu^2_{j_1,j_2,q}(P)) \subset \bigcup_{i = j_1 + 1}^{j_2-1}U_{i,q + 1}^1\cup \bigcup_{i = j_1}^{j_2-1}U_{i,q + 1}^2 \cup U_{j_2}^3
\]
and the following estimates hold:
\begin{align}
\label{mixed12}\nu^2_{j_1,j_2,q}(P)(U^2_{j_1,q + 1}) &\ge 1 - \frac{C}{2^{j_2}},\\
\label{mixed22}\nu^2_{j_1,j_2,q}(P)\left(\bigcup_{i = j_1 + 1}^{j_2-1}U_{i,q + 1}^1\cup \bigcup_{i = j_1 + 1}^{j_2-1}U_{i,q + 1}^2 \cup U_{j_2}^3\right) &\le \frac{C}{2^{j_2}},\\
\label{mixed32}\nu^2_{j_1,j_2,q}(P)(U_{j_2}^3) &> 0.
\end{align}
\end{prop}
\begin{proof}
The proof is analogous to the one of Proposition \ref{comb1}, and is based on a combination of Lemma \ref{costrone3} and Proposition \ref{costrlam}. We omit the details.
\end{proof}

\section{The inductive proposition and conclusion}\label{sec:ind}

This Section is devoted to the proof of Theorem \ref{tmain}. As in every convex integration-type argument, the construction of the exact solution is inductive. Let $\Omega \subset \R^2$ be any convex, bounded and open domain. Define countably many families of sets
\[
\mathcal{F}_n \doteq \{\Omega_{1,n},\dots, \Omega_{N_n,n}\},
\]
with the following properties:
\begin{itemize}
\item for all $n$, $\{\Omega_{j,n}\}_j$ are pairwise disjoint, open sets, whose union gives $\Omega$ up to a set of zero measure;
\item for every $j \in \{1,\dots, N_{n + 1}\}$, if $\Omega_{j,n+1}\cap \Omega_{k(j),n} \neq \emptyset$ for some $k(j) \in \{1,\dots, N_n\}$, then $\Omega_{j,n + 1} \subset \Omega_{k(j), n}$;
\item $\diam(\Omega_{j,n}) \le \frac{1}{n}$, $\forall j \in \{1,\dots, N_n\}$, $\forall n \in \N$.
\end{itemize}
We fixed at the end of Section \ref{open} the parameter $c$ so that \eqref{incre}-\eqref{increless} and Proposition \ref{ansmag2} hold. Consequently, we let $I$ be an index for which Propositions \ref{ansmag2}-\ref{comb1}-\ref{comb2} and Lemmas \ref{sigma}-\ref{costrone3}-\ref{base} hold, and $t_0$ is fixed by Lemma \ref{sigma}. Furthermore, let as usual $Q(c)$ be the open set of parameters of our constructions, whose points are denoted by $P$. Further, we will always denote
\begin{equation}\label{gamma}
\gamma \doteq 1 + \min\{1,2(p-1)\}.
\end{equation}
The inductive construction will be formalized in Proposition \ref{induc}, and we illustrate here its first step. Start with any affine map $w_0 = Mx$ on $\Omega$, where $M \in A_I(Q(c))$, say $M = A_I(P)$, $P \in Q(c)$. We write $A_I(P)$ as the barycenter of the laminate of finite order $\nu_{I,I}(P)$ of Lemma \ref{base}. In every subset $\Omega_{j,I} \in \mathcal{F}_{I}$, use Proposition \ref{ind} to find a piecewise affine Lipschitz map $w_1$ with the following properties
\begin{enumerate}
\item $w_1 = Mx$ on $\partial \Omega$;
\item $\|w_1 - w_0\|_{L^\infty(\Omega, \R^2)}\le \frac{1}{2}$;
\item\label{open11} $Dw_1 \in V_{1}$ a.e.;
\item\label{OPEN111} the following estimates hold, for all $j \in \{1,\dots, N_1\}$:
\begin{enumerate}
\item $\displaystyle k_1\frac{|\Omega_{j,I}|}{I} \le |\{x \in \Omega_{j,I}: Dw_1(x) \in U_{I ,I}^1\}| \le k_2\frac{|\Omega_{j,I}|}{I}$;
\item $\displaystyle k_1\frac{|\Omega_{j,I}|}{I} \le |\{x \in \Omega_{j,I}: Dw_1(x) \in U_{I,I}^2\}| \le k_2\frac{|\Omega_{j,I}|}{I}$;
\item $0 < |\{x \in \Omega_{j,I}: Dw_1(x) \in U_{I+1}^3\}| \le e^{\frac{C}{I^\gamma}}e^{-2\frac{\min\{G_p(c),G_p(2c)\}}{I}}|\Omega_{j,I}|$.
\end{enumerate}
\end{enumerate}
Notice that \eqref{open11}-\eqref{OPEN111} are consequences of the openness of $V_1$, $U_{I,I}^1,U_{I,I}^2$ and $U_{I + 1}^3$. We can now move on to the inductive Proposition. Recall that
\[
V_{n} = \bigcup_{i = I}^{I + n-1} U_{i,I + n-1}^1\cup \bigcup_{i = I}^{I + n-1} U_{i,I +n-1}^2 \cup U_{I + n}^3.
\]
In the proof, we will consider the usual positive and radial mollifier $\rho \in C^\infty_c(B_1)$, and its associated mollification kernel $\rho_\delta$.
\begin{prop}\label{induc}
Let $n \ge 1$. Suppose we are given piecewise affine, Lipschitz maps $\{w_1,\dots, w_n\}$, $w_q: \Omega \to \R^2$, $\forall q \in \{1,\dots, n\}$, and positive decreasing numbers $\delta_1,\dots, \delta_n$, with $\delta_n \le 2^{-n}$ enjoying the following properties:
\begin{enumerate}[(1)]
\item \label{hp1} $w_q = Mx$ on $\partial\Omega$;
\item \label{hp2} $\|w_q\star \rho_{\delta_q} - w_q\|_{W^{1,1}(\R^2,\R^2)} \le 2^{-q}$, $\forall q \in \{1,\dots, n\}$;
\item \label{hp3} $\|w_{q+ 1} - w_q\|_{L^\infty(\R^2,\R^2)} \le \delta_{q}2^{-q}$, $\forall q \in \{1,\dots, n-1\}$;
\item \label{hp4} $Dw_q \in V_{ q}$, $\forall q \in \{1,\dots, n\}$.
\end{enumerate}
Then, there exist a piecewise affine, Lipschitz map $w_{n + 1}$ and a number $0 <\delta_{n + 1} < \min\{\delta_n,2^{-n -1}\}$, with the following properties:
\begin{enumerate}[(i)]
\item \label{th1} $w_{n + 1} = Mx$ on $\partial\Omega$;
\item \label{th2} $\|w_{n + 1}\star \rho_{\delta_{n + 1}} - w_{n + 1}\|_{W^{1,1}(\R^2,\R^2)} \le 2^{-n - 1}$;
\item \label{th3} $\|w_{n+ 1} - w_n\|_{L^\infty(\R^2,\R^2)} \le \delta_{n}2^{-n-1}$;
\item \label{th4} $Dw_{n + 1} \in V_{n+1}$;
\item \label{th5} The following estimates hold, for $k =1,2, i \in \{0,\dots, n-1\}, j \in \{1,\dots, N_{n + 1}\}$:
\begin{equation}\label{finalkbasso}
\begin{split}
\left(1 - \frac{C}{2^{n}}\right)&|\{x \in \Omega_{j,n+1}: Dw_{n} \in U^k_{I + i, I + n-1}\}| \le |\{x \in \Omega_{j,n+1}: Dw_{n + 1} \in U^k_{I + i, I + n}\}|\\
&\le|\{x \in \Omega_{j,n+1}: Dw_{n} \in U^k_{I + i, I + n-1}\}| + \frac{Cn}{2^{n}}|\Omega_{j,n + 1}|.
\end{split}
\end{equation}
\begin{equation}\label{finalkalto}
\begin{split}
\frac{k_1}{n}&|\{x \in \Omega_{j,n+1}: Dw_{n} \in U^3_{I + n}\}| \le|\{x \in \Omega_{j,n+1}: Dw_{n + 1} \in U^k_{I + n, I + n}\}|\\
&\le \frac{k_2}{n}|\{x \in \Omega_{j,n+1}: Dw_{n} \in U^3_{I + n}\}| + \frac{Cn}{2^{n}}|\Omega_{j,n + 1}|,
\end{split}
\end{equation}
and finally
\begin{equation}\label{finalU3}
\begin{split}
0 &< |\{x \in \Omega_{j,n+1}: Dw_{n + 1} \in U^3_{I + n + 1}\}|\\
&\qquad\le e^{\frac{C}{n^\gamma}}e^{\frac{-2 \min\{G_p(c),G_p(2c)\}}{I +n}}|\{x \in \Omega_{j,n+1}: Dw_{n} \in U^3_{I + n}\}| + \frac{Cn}{2^{n}}|\Omega_{j,n + 1}|.
\end{split}
\end{equation}
\end{enumerate}
\end{prop}
Notice that in the previous statement, in order to write quantities like
\[
 \|w_{n + 1}\star \rho_{\delta_{n + 1}} - w_{n + 1}\|_{W^{1,1}(\R^2,\R^2)},
\]
we assume to have extended all maps $w_q$ as $w_q = Mx$ outside $\Omega$. By \eqref{hp1}-\eqref{th1}, this extension preserves the Lipschitzianity of those maps, hence we are allowed to do so.
\begin{proof}
Let $\Omega' \subset \Omega_{j,n+1} \subset \Omega_{k(j),n}$ be an open set where $w_n$ is affine, i.e. $w_n(x) = S x +b$ for some $S \in V_n$, $b \in \R^2$. By Lemma \ref{sigma}, $V_n$ is a disjoint union, hence there are only three cases: either
\begin{equation}\label{case1}
S \in \bigcup_{i = I}^{I + n-1} U_{i,I + n-1}^1,
\end{equation}
\begin{equation}\label{case2}
S \in\bigcup_{i = I}^{I + n-1} U_{i,I +n-1}^2,
\end{equation}
or
\begin{equation}\label{case3}
S \in U_{I + n}^3.
\end{equation}
Assume \eqref{case1} holds, and in particular that $S \in U^1_{I + i,I+n-1}$ for some $i \in \{0,\dots, n-1\}$. By definition, there exists $P \in Q(c)$ such that $$S = \Phi^1_{I + i, t_{I + n-1}}(P).$$ We can then use the laminate $\nu_{I+i,I+n+1,I+n-1}^1(P)$ of Proposition \ref{comb1} combined with Proposition \ref{ind} to find a Lipschitz and piecewise affine map $f$ with the following properties
\begin{itemize}
\item $f|_{\partial \Omega'} = w_n|_{\partial\Omega'}$;
\item Almost everywhere on $\Omega'$, $$Df \in \bigcup_{j = I + i}^{I + n} U_{j,I + n}^1\cup \bigcup_{j = I + i}^{I + n} U_{j,I +n}^2 \cup U_{I + n+1}^3$$ with the estimates:
\begin{equation}\label{est1fin}
\left(1 - \frac{C}{2^{n}}\right)|\Omega'| \le |\{x \in \Omega': Df \in U_{I +i,I + n}^1\}| \le |\Omega'| ,
\end{equation}
\begin{equation}\label{est2fin}
\left|\left\{x \in \Omega': Df \in \bigcup_{j = I + i + 1}^{I + n} U_{j,I + n}^1\cup \bigcup_{j = I + i}^{I + n} U_{j,I +n}^2 \cup U_{I + n+1}^3\right\}\right| \le \frac{C}{2^{n}}|\Omega'|,
\end{equation}
\begin{equation}\label{basso1}
0 <\left|\left\{x \in \Omega': Df \in U_{I + n + 1}^3\right\}\right|;
\end{equation}
\item $\|f- w_n\|_{L^\infty}(\Omega') \le \delta_n2^{-n -1 }$.
\end{itemize}
On such a set $\Omega'$, we replace $w_n|_{\Omega'}$ by $f$. Analogously, if \eqref{case2} holds, i.e. $S \in U^2_{I + i,I+n-1}$ for some $i \in \{0,\dots, n-1\}$, then, by definition, $$S = \Phi^2_{I + i, t_{I + n-1}}(P)$$ for some $P \in Q(c)$, and we can then use the laminate $\nu^2_{I+i,I+n+1,I+n-1}(P)$ of Proposition \ref{comb2} combined with Proposition \ref{ind} to find a Lipschitz and piecewise affine map $g$ with the following properties
\begin{itemize}
\item $g|_{\partial \Omega'} = w_n|_{\partial\Omega'}$;
\item Almost everywhere on $\Omega'$, $$Dg \in \bigcup_{j = I + i + 1}^{I + n} U_{j,I + n}^1\cup \bigcup_{j = I + i}^{I + n} U_{j,I +n}^2 \cup U_{I + n+1}^3$$ with the estimates:
\begin{equation}\label{est3fin}
\left(1 - \frac{C}{2^{n}}\right)|\Omega'| \le |\{x \in \Omega': Dg \in U_{I +i,I + n}^2\}| \le |\Omega'| ,
\end{equation}
\begin{equation}\label{est4fin}
\left|\left\{x \in \Omega': Dg \in \bigcup_{j = I + i + 1}^{I + n} U_{j,I + n}^1\cup \bigcup_{j = I + i+1}^{I + n} U_{j,I +n}^2 \cup U_{I + n+1}^3\right\}\right| \le \frac{C}{2^{n}}|\Omega'|,
\end{equation}
and
\begin{equation}\label{basso2}
0 <\left|\left\{x \in \Omega': Dg \in U_{I + n + 1}^3\right\}\right|;
\end{equation}
\item $\|g- w_n\|_{L^\infty}(\Omega') \le \delta_n2^{-n -1 }$.
\end{itemize}
On such a set $\Omega'$, we replace $w_n|_{\Omega'}$ by $g$. Finally, in the third case, \eqref{case3}, we have $S = A_{I + n}(P)$ for some $P \in Q(c)$. We use Lemma \ref{base} combined with Proposition \ref{ind} to find a Lipschitz and piecewise affine map $h$ with the following properties
\begin{itemize}
\item $h|_{\partial \Omega'} = w_n|_{\partial\Omega'}$;
\item Almost everywhere on $\Omega'$,
\[
Dh \in U_{I + n,I + n}^1 \cup  U_{I + n,I + n}^2 \cup U_{I + n + 1}^3,
\]
with the estimates:
\begin{equation}\label{est5fin}
\frac{k_1}{n}|\Omega'| \le |\{x \in \Omega': Dh \in U_{I +n,I + n}^1\}| \le \frac{k_2}{n}|\Omega'|,
\end{equation}
\begin{equation}\label{est6fin}
\frac{k_1}{n}|\Omega'| \le |\{x \in \Omega': Dh \in U_{I +n,I + n}^2\}| \le \frac{k_2}{n}|\Omega'|
\end{equation}
and
\begin{equation}\label{est7fin}
0 < |\{x \in \Omega': Dh \in U_{I +n + 1}^3\}| \le e^{\frac{C}{n^\gamma}}e^{\frac{-2 \min\{G_p(c),G_p(2c)\}}{I +n}}|\Omega'|
\end{equation}
\item $\|h- w_n\|_{L^\infty}(\Omega') \le \delta_n2^{-n-1}$.
\end{itemize}
Finally, for a set $\Omega'$ of this type, we replace $w_n|_{\Omega'}$ by $h$. These replacements precisely give $w_{n + 1}$. Notice that in all of the above estimates the quantities $2^n$, $n$ and $e^{\frac{C}{n^\gamma}}$ should have been $2^{n + I + 1}$, $n + I + 1$ and $e^{\frac{C}{(I + n)^\gamma}}$, but, since they are all comparable, we can simply reabsorb the errors in making these substitutions inside the constants $C$, $k_1,k_2$.
\\
\\By construction, \eqref{th1}-\eqref{th3}-\eqref{th4} hold. The existence of $\delta_{n + 1}$ as in $\eqref{th2}$ is guaranteed by the Lipschitzianity of $w_{n + 1}$. We can now show the estimates asserted in \eqref{th5}. Let first $i \in \{0,\dots, n-1\}$, $j \in \{1,\dots, N_{n + 1}\}$. Then:
\begin{equation*}\label{PHI1basso}
\begin{split}
\left(1 - \frac{C}{2^{n}}\right)&|\{x \in \Omega_{j,n+1}: Dw_{n} \in U^1_{I + i, I + n-1}\}| \overset{\eqref{est1fin}}{\le} |\{x \in \Omega_{j,n+1}: Dw_{n + 1} \in U^1_{I + i, I + n}\}|\\
&\overset{\eqref{est1fin}-\eqref{est2fin}-\eqref{est4fin}}{\le}|\{x \in \Omega_{j,n+1}: Dw_{n} \in U^1_{I + i, I + n-1}\}| + \sum_{\ell = 0}^{i - 1}\frac{C}{2^{n}}|\{x \in \Omega_{j,n+1}: Dw_{n} \in U^1_{I + \ell, I + n-1}\}| \\
&\qquad+  \sum_{\ell = 0}^{i - 1}\frac{C}{2^{n}}|\{x \in \Omega_{j,n+1}: Dw_{n} \in U^2_{I + \ell, I + n-1}\}|\\
&\le |\{x \in \Omega_{j,n+1}: Dw_{n} \in U^1_{I + i, I + n-1}\}| + \frac{Cn}{2^{n}}|\Omega_{j,n + 1}|.
\end{split}
\end{equation*}
With analogous computations, for the same ranges of $i$ and $j$, we find
\begin{equation*}\label{PHI2basso}
\begin{split}
\left(1 - \frac{C}{2^{n}}\right)&|\{x \in \Omega_{j,n+1}: Dw_{n} \in U^2_{I + i, I + n-1}\}| \overset{\eqref{est3fin}}{\le} |\{x \in \Omega_{j,n+1}: Dw_{n + 1} \in U^2_{I + i, I + n}\}|\\
&\overset{\eqref{est3fin}-\eqref{est2fin}-\eqref{est4fin}}{\le}|\{x \in \Omega_{j,n+1}: Dw_{n} \in U^2_{I + i, I + n-1}\}| + \frac{Cn}{2^{n}}|\Omega_{j,n + 1}|.
\end{split}
\end{equation*}
We still need to write the estimates for $i = n$:
\begin{equation*}\label{PHI1alto}
\begin{split}
\frac{k_1}{n}&|\{x \in \Omega_{j,n+1}: Dw_{n} \in U^3_{I + n}\}| \overset{\eqref{est5fin}}{\le} |\{x \in \Omega_{j,n+1}: Dw_{n + 1} \in U^1_{I + n, I + n}\}|\\
&\overset{\eqref{est5fin}-\eqref{est2fin}-\eqref{est4fin}}{\le} \frac{k_2}{n}|\{x \in \Omega_{j,n+1}: Dw_{n} \in U^3_{I + n}\}| + \frac{Cn}{2^{n}}|\Omega_{j,n + 1}|,
\end{split}
\end{equation*}
and
\begin{equation*}\label{PHI2alto}
\begin{split}
\frac{k_1}{n}&|\{x \in \Omega_{j,n+1}: Dw_{n} \in U^3_{I + n}\}| \overset{\eqref{est5fin}}{\le} |\{x \in \Omega_{j,n+1}: Dw_{n + 1} \in U^2_{I + n, I + n}\}|\\
&\overset{\eqref{est6fin}-\eqref{est2fin}-\eqref{est4fin}}{\le} \frac{k_2}{n}|\{x \in \Omega_{j,n+1}: Dw_{n} \in U^3_{I + n}\}| + \frac{Cn}{2^{n}}|\Omega_{j,n + 1}|.
\end{split}
\end{equation*}
Finally,
\begin{equation*}\label{U3}
\begin{split}
0 &\overset{\eqref{basso1}-\eqref{basso2}-\eqref{est7fin}}{<} |\{x \in \Omega_{j,n+1}: Dw_{n + 1} \in U^3_{I + n + 1}\}|\\
&\overset{\eqref{est7fin}-\eqref{est2fin}-\eqref{est4fin}}{\le} e^{\frac{C}{n^\gamma}}e^{\frac{-2 \min\{G_p(c),G_p(2c)\}}{I +n}}|\{x \in \Omega_{j,n+1}: Dw_{n} \in U^3_{I + n}\}| + \frac{Cn}{2^{n}}|\Omega_{j,n + 1}|.
\end{split}
\end{equation*}
This concludes the proof of the inductive Proposition.
\end{proof}

We now show that the sequence $\{w_n\}$ is equibounded in $W^{1,1 + \eps}(\Omega,\R^2)$, for some $\eps = \eps(p) > 0$.

\begin{prop}\label{equib}
The sequence $\{w_n\}$ is equibounded in $W^{1,1 + \eps}(\Omega,\R^2)$.
\end{prop}

We first show the following:
\begin{lemma}\label{upest}
There exists a constant $C = C(p) > 0$ such that the following estimates hold for all $n\ge 1$, $0\le i \le n-1$,
\begin{align}
\label{afterind1}|\{x \in \Omega: Dw_{n} \in U^3_{I + n}\}| &\le Cn^{-2 \min\{G_p(c),G_p(2c)\}},\\
\label{afterind2}|\{x \in \Omega: Dw_{n} \in U^1_{I+ i,I + n - 1}\}| &\le Ci^{-2 \min\{G_p(c),G_p(2c)\} - 1},\\
\label{afterind3}|\{x \in \Omega: Dw_{n} \in U^2_{I+ i,I + n - 1}\}| &\le Ci^{-2 \min\{G_p(c),G_p(2c)\} - 1}.
\end{align}
\end{lemma}
\begin{proof}
Sum over $j \in \{1,\dots, N_{n + 1}\}$ inequality \eqref{finalU3} to find
\[
|\{x \in \Omega: Dw_{n + 1} \in U^3_{I + n + 1}\}|\le e^{\frac{C}{n^\gamma}}e^{\frac{-2\min\{ G_p(c),G_p(2c)\}}{I +n}}|\{x \in \Omega: Dw_{n} \in U^3_{I + n}\}| + \frac{Cn}{2^{n}}
\]
Now we can apply this relation recursively to discover that, for all $n \in \N$,
\[
|\{x \in \Omega: Dw_{n} \in U^3_{I + n}\}| \le C\left(\prod_{\ell = 1}^{n-1}e^{\frac{C}{\ell^\gamma}}e^{\frac{-2 \min\{G_p(c),G_p(2c)\}}{I +\ell}} + \sum_{\ell = 1}^{n-1}\left(\prod_{j = \ell + 1}^{n-1}e^{\frac{C}{j^\gamma}}e^{\frac{-2\min\{G_p(c),G_p(2c)\}}{I +j}}\right)\frac{\ell}{2^\ell} \right).
\]
In the previous expression, we used the convention introduced in \eqref{conven}. Using \eqref{Euler}, we can estimate
\[
\prod_{\ell = 1}^{n-1}e^{\frac{-2 \min\{G_p(c),G_p(2c)\}}{I +\ell}} = e^{-2\min\{G_p(c),G_p(2c)\}\sum_{\ell = 1}^{n-1}\frac{1}{I + \ell}} \le C n^{-2\min\{G_p(c),G_p(2c)\}}.
\]
On the other hand, $\gamma > 1$ by \eqref{gamma}. Therefore, $\prod_{\ell = 1}^ne^{\frac{C}{\ell^\gamma}}$ is uniformly bounded in $n$. Concerning the second addendum, we start by estimating for all $\ell + 1 \le n-1$:
\begin{align*}
\prod_{j = \ell + 1}^{n-1}e^{\frac{C}{j^\gamma}}e^{\frac{-2 \min\{G_p(c),G_p(2c)\}}{I +j}} &\le e^{\sum_{\ell = 1}^{\infty}\frac{C}{j^\gamma}}e^{\sum_{j = \ell + 1}^{n-1}\frac{-2 \min\{G_p(c),G_p(2c)\}}{I +j}} \\
&\overset{\eqref{Euler}}{\le} Ce^{-2 \min\{G_p(c),G_p(2c)\}\ln\left(\frac{I + n - 1}{I +\ell + 1}\right)} \le C\left(\frac{I + \ell + 1}{I + n - 1}\right)^{2\min\{G_p(c),G_p(2c)\}}.
\end{align*}
Thus,
\begin{align*}
\sum_{\ell = 1}^{n-1}\prod_{j = \ell + 1}^{n-1}e^{\frac{C}{j^\gamma}}e^{\frac{-2 \min\{G_p(c),G_p(2c)\}}{I +j}} \frac{\ell}{2^\ell} &\le C\left(\frac{1}{I + n - 1}\right)^{2\min\{G_p(c),G_p(2c)\}}\sum_{\ell = 1}^{n-1}(I + \ell + 1)^{2\min\{G_p(c),G_p(2c)\}}\frac{\ell}{2^\ell} \\
&\le C n^{-2\min\{G_p(c),G_p(2c)\}},
\end{align*}
which concludes the proof of \eqref{afterind1}. Using \eqref{afterind1} and \eqref{finalkalto}, for $k = 1,2$ and for all $n \in \N$, we have:
\begin{equation}\label{nn}
|\{x \in \Omega: Dw_{n} \in U^k_{I + n-1, I + n-1}\}| \le Cn^{-2\min\{G_p(c),G_p(2c)\} - 1} + \frac{Cn}{2^{n}}.
\end{equation}
Through \eqref{finalkbasso} and a simple inductive reasoning, we also find, for all $n \ge 1, 0 \le i \le n- 1$,
\begin{align*}
|\{x \in \Omega: Dw_{n + 1} \in &U^k_{I + i, I + n}\}|\le|\{x \in \Omega: Dw_{n} \in U^k_{I + i, I + n-1}\}| + \frac{Cn}{2^{n}} \\
&\overset{\eqref{nn}}{\le} C(i + 1)^{-2\min\{G_p(c),G_p(2c)\} -1} + C\sum_{\ell = i + 1}^\infty\frac{\ell}{2^{\ell}} \le C(i + 1)^{-2\min\{G_p(c),G_p(2c)\} -1}.
\end{align*}
This concludes the proof.
\end{proof}
We can now prove Proposition \ref{equib}.
\begin{proof}[Proof of Proposition \ref{equib}]
Let $\Omega_n \doteq  \{x \in \Omega: Dw_{n} \in U^3_{I + n}\}$, $\Omega^1_{i,n}\doteq \{x \in \Omega: Dw_{n} \in U^1_{I+ i,I + n - 1}\}$, $\Omega^2_{i,n}\doteq \{x \in \Omega: Dw_{n} \in U^2_{I+ i,I + n - 1}\}$, for $i \in \{0,\dots, n-1\}$. Notice that, by definition:
\begin{align*}
\sup\{|M|: M \in U_{I + i,I +n - 1}^1\} &= \sup_{P \in Q(c)}|\Phi^1_{I + i,t_{I + n - 1}}(P)|,\\
\sup\{|M|: M \in U_{I + i,I +n - 1}^2\} &= \sup_{P \in Q(c)}|\Phi^2_{I + i,t_{I + n - 1}}(P)|,\\
\sup\{|M|: M \in U^3_{I +n}\} &= \sup_{P \in Q(c)}|A_{I + n}(P)|.
\end{align*}
Thus we can estimate, for any $q \in [1,\infty)$,
\begin{equation}\label{q}
\begin{split}
\|Dw_n\|_{L^q(\Omega,\R^2)}^q &\le \sum_{i = 1}^{n} \sup_{P \in Q(c)}|\Phi^1_{I + i,t_{I + n  - 1}}(P)|^q|\Omega_{i,n}^1| +\sum_{i = 1}^{n} \sup_{P \in Q(c)}|\Phi^2_{I + i,t_{I + n - 1}}(P)|^q|\Omega_{i,n}^2| \\
&+\sup_{P \in Q(c)}|A_{I + n}(P)|^q|\Omega_n|.
\end{split}
\end{equation}
Recall that by Proposition \ref{ansmag2}\eqref{b1},
\[
\sup_{P \in \overline{Q(c)}}|z_i(P)| \le Ci^{2(p-2)}, \quad \forall i \ge 1.
\]
Through \eqref{vwi}, we see that also
\[
\sup_{P \in \overline{Q(c)}}|v_i(P)| \le Ci^{2(p-2)}, \quad \forall i \ge 1.
\]
Using the definitions of $\Phi^1_{i,t}$ and $\Phi^2_{i,t}$ and the previous estimates, we can bound
\begin{align*}
\sup_{P \in Q(c)}|\Phi^1_{I + i,t_{I + n + 1}}(P)| &\le Ci^{\max\{2,2(p-1)\}},\quad \sup_{P \in Q(c)}|\Phi^2_{I + i,t_{I + n + 1}}(P)| \le Ci^{\max\{2,2(p-1)\}},\\
& \sup_{P \in Q(c)}|A_{I + n + 1}(P)| \le Cn^{\max\{2,2(p-1)\}}.
\end{align*}
Combining the latter with the estimates of Lemma \ref{upest}, we can continue \eqref{q} as
\begin{align*}
\|Dw_n\|_{L^q(\Omega,\R^2)}^q \le C\sum_{i = 1}^ni^{q\max\{2,2(p-1)\}}i^{-2\min\{G_p(c),G_p(2c)\}-1} + Cn^{q\max\{2,2(p-1)\}}n^{-2\min\{G_p(c),G_p(2c)\}}.
\end{align*}
Our choices \eqref{incre} and \eqref{increless} imply \eqref{eqn:choice-c}, which yields precisely
\[
\max\{1,p-1\} - \min\{G_p(c),G_p(2c)\} < 0.
\]
Therefore, for some $q >1$, we also have
\[
q\max\{1,p-1\} - \min\{G_p(c),G_p(2c)\} < 0,
\]
which proves the present proposition.
\end{proof}
We are finally in position to prove Theorem \ref{tmain}, that we restate here:
\begin{theorem}\label{tmaintext}
Let $\Omega \subset \R^2$ be a ball. For every $p \in (1,\infty), p \neq 2$, there exists $\eps = \eps(p) > 0$ and a continuous $u \in W^{1,p-1 + \eps}(\Omega)$ such that $u$ is affine on $\partial\Omega$,
\begin{equation}\label{dery1}
\frac{3}{4} \le \partial_yu \le \frac{5}{4}, \quad \text{a.e. on }\Omega,
\end{equation}
\begin{equation}\label{conv1}
\dv(|Du|^{p-2}Du) = 0
\end{equation}
in the sense of distributions, but for all open $B \subset \Omega$
\begin{equation}\label{div1}
\int_{B}|Du|^pdx = + \infty.
\end{equation}
\end{theorem}
\begin{proof}
Let $\{w_n\}$ be the sequence constructed in Proposition \ref{induc}. By \eqref{th3}, we find that there exists a limit $w_\infty = \lim_{n\to \infty}w_n$ in the $L^\infty(\R^2,\R^2)$ topology. Since every $w_n$ is continuous, then so is $w_\infty$. By Proposition \ref{equib}, we also have that $w_\infty$ is the weak limit in $W^{1,1+\eps}(\Omega,\R^2)$ of $w_\infty$. We shall now prove that the convergence is strong. This is a crucial but standard point of this type of constructions, see \cite[Theorem 5.3]{DMU}. Indeed, using the notation of Proposition \ref{induc}, we have for all $n$:
\begin{align*}
\|Dw_n - Dw_\infty\|_{L^1} \le \|Dw_n - Dw_n\star \rho_{\delta_n}\|_{L^1}+ \|Dw_\infty - Dw_\infty\star \rho_{\delta_n}\|_{L^1}+ \|Dw_\infty\star \rho_{\delta_n} - Dw_n\star \rho_{\delta_n}\|_{L^1}.
\end{align*}
By our choice of $\delta_n$, see Proposition \ref{ind}\eqref{th2}, and the fact that $Dw_\infty \in W^{1,1 + \eps}$, we have that the first two addenda converge to $0$. Concerning the third, we can employ standard estimates on mollification to bound:
\[
\|Dw_\infty\star \rho_{\delta_n} - Dw_n\star \rho_{\delta_n}\|_{L^1} \le \frac{C}{\delta_n}\|w_\infty - w_n\|_{L^\infty} \le \frac{C}{\delta_n}\sum_{j = n}^\infty\|w_{j+1}- w_j\|_{L^\infty} \overset{\eqref{th3}}{\le} \frac{C}{\delta_n}\sum_{j = n}^\infty\delta_j2^{-j - 1} \le C\sum_{j = n}^\infty 2^{-j - 1}.
\]
We infer the strong convergence of $Dw_n$ to $Dw_\infty$. Up to passing to a subsequence, this yields a subsequence of $\{Dw_n\}$ that converges pointwise a.e.. Through Proposition \ref{induc}\eqref{th4} and Lemma \ref{conv}, we deduce, for almost all $x \in \Omega$,
\[
Dw_\infty(x) \in K_p.
\]
Proposition \ref{prop:translation} tells us that, if $w_\infty = (w^1_\infty,w^2_\infty)$, then $u \doteq w_\infty^1$, has the right integrability and fulfills \eqref{conv1}. We now turn to \eqref{dery1}, which is straightforward: by definition of $V_n$ it follows that if
\[
X = \left(\begin{array}{cc}x_{11} & x_{12} \\ x_{21} & x_{22}\end{array}\right) \in V_n
\]
for some $n$, then $\frac{3}{4} < x_{12} < \frac{5}{4}$. Since by Proposition \ref{ind}\eqref{th4} we have, a.e. on $\Omega$,
\[
Dw_n \in V_n, \quad \forall n,
\]
it then follows, if $w_n = (w_n^1,w_n^2)$, that a.e. on $\Omega$
\[
\frac{3}{4} \le \partial_{y}w_n^1 \le \frac{5}{4}.
\]
Therefore, $w_\infty$ enjoys the same property. We shall now show \eqref{div1}. In order to do so, we show that $Dw_\infty$ is (essentially) discontinuous on any open subset $B \subset \Omega$. We claim this is enough to conclude \eqref{div1}. Indeed, suppose by contradiction that $Du \in L^p(\Omega')$, for some open $\Omega' \subset \Omega$. Then, since $u$ solves \eqref{conv1}, we see that it is a \emph{weak} solution of the $p$-Laplace equation, and by the classical regularity theory for the $p$-Laplace equation, it follows that $Du$ is continuous in $\Omega'$, and hence so is $Dw_\infty$, which would result in a contradiction. Thus, we conclude the proof of the present theorem by showing this last assertion. Fix $B \subset \Omega$ open. Since 
\[
\sup_{j}\{\diam(\Omega_{j,n}): \Omega_{j,n} \in \mathcal{F}_n\} \le \frac{1}{n},
\]
we find $n_0$ and $j_0 \in \{1,\dots, N_{n_0}\}$ such that $\Omega_{j_0,n_0} \subset B$. Recall that, by Lemma \ref{conv} and the pointwise convergence of a subsequence of $\{Dw_n\}_n$,
\[
Dw_\infty \in \bigcup_{i\ge I}B_i(Q(c))\cup \bigcup_{i \ge I}D_i(Q(c)).
\]
Our aim is to show that
\begin{equation}\label{BIset}
|\{x \in \Omega_{j_0,n_0}: Dw_\infty(x) \in \bigcup_{i\ge I}B_i(Q(c))\}| > 0
\end{equation}
and
\begin{equation}\label{DIset}
|\{x \in \Omega_{j_0,n_0}: Dw_\infty(x) \in \bigcup_{i\ge I}D_i(Q(c))\}| > 0.
\end{equation}
Since, by Lemma \ref{sigma}, $\bigcup_{i\ge I}B_i(Q(c)) \cap \bigcup_{i\ge I}D_i(Q(c)) = \emptyset$, this would prove that $Dw_\infty$ is not continuous in $B$. We claim that, in order to prove \eqref{BIset}-\eqref{DIset}, it is sufficient to show that there exists a constant $c_0 = c_0(j_0,n_0) > 0$ such that for all $k= 1,2$ and $n \ge n_0$,
\begin{equation}\label{frombelow}
|\{x \in \Omega_{j_0,n_0}: Dw_n(x) \in \bigcup_{i = I}^{I + n-1} U_{i,I + n-1}^k\}| \ge c_0.
\end{equation}
Indeed, let $\Omega_n^k \doteq \{x \in \Omega_{j_0,n_0}: Dw_n(x) \in \bigcup_{i = I}^{I + n-1} U_{i,I + n-1}^k\}$. If \eqref{frombelow} holds, then
\begin{align*}
\int_{\Omega_{j_0,n_0}}\dist\left(Dw_n(x),\bigcup_{i = I}^{n-1}B_i(Q(c))\right)dx \ge \int_{\Omega_n^2}\dist\left(Dw_n(x),\bigcup_{i = I}^{n-1}B_i(Q(c))\right)dx \ge c'c_0,
\end{align*}
where in the last inequality we used that
\[
\dist\left(M,\bigcup_{i = I}^{n-1}B_i(Q(c))\right) \ge c' > 0, \quad \forall M \in \ \bigcup_{i = I}^{I + n-1} U_{i,I + n-1}^2,
\]
as can be easily seen by properties \eqref{primo}-\eqref{secondo}. By the strong convergence $Dw_n \to Dw_\infty$ in $L^1(\Omega)$, it follows that
\[
\int_{\Omega_{j_0,n_0}}\dist\left(Dw_\infty(x),\bigcup_{i = I}^{\infty}B_i(Q(c))\right)dx \ge c'c_0 > 0,
\]
and we conclude \eqref{DIset}. Analogously, one infers \eqref{BIset} from \eqref{frombelow} for $k = 1$. Hence, we only need to prove \eqref{frombelow}. Let any $n > n_0$. Recall that we chose the partitions $\{\Omega_{j,n}\}$ with the property that for every $j \in \{1,\dots, N_{n + 1}\}$, if $\Omega_{j,n+1}\cap \Omega_{k(j),n} \neq \emptyset$ for some $k(j) \in \{1,\dots, N_n\}$, then $\Omega_{j,n + 1} \subset \Omega_{k(j), n}$. Therefore, we can sum over a suitable subset of indexes $j \in \{1,\dots, N_{n}\}$ to rewrite, for all $n \ge n_0 + 2$, $i = n_0 +1$ and $k = 1,2$, the bounds from below of \eqref{finalkbasso}-\eqref{finalkalto}-\eqref{finalU3} as
\begin{align}
\label{below1} &\left(1 - \frac{C}{2^{n}}\right)|\{x \in \Omega_{j_0,n_0}: Dw_{n} \in U^k_{I + n_0 + 1, I + n-1}\}| \le |\{x \in\Omega_{j_0,n_0}: Dw_{n + 1} \in U^k_{I + n_0+1, I + n}\}|,\\
\label{below2}&\frac{k_1}{n}|\{x \in \Omega_{j_0,n_0}: Dw_{n} \in U^3_{I + n}\}| \le|\{x \in \Omega_{j_0,n_0}: Dw_{n + 1} \in U^k_{I + n, I + n}\}|,\\
\label{below3}&0 < |\{x \in \Omega_{j_0,n_0}: Dw_{n-1} \in U^3_{I + n-1}\}|.
\end{align}
We use \eqref{below1} inductively to find that, if $n \ge n_0 + 3$,
\begin{equation}\label{last1}
\prod_{i = n_0 + 2}^{n-1}\left(1 - \frac{C}{2^{i}}\right)|\{x \in \Omega_{j_0,n_0}: Dw_{n_0+ 2} \in U^k_{I + n_0+ 1, I + n_0 + 1}\}|\le |\{x \in \Omega_{j_0,n_0}: Dw_{n} \in U^k_{I + n_0+ 1, I + n-1}\}|.
\end{equation}
We also have
\begin{equation}\label{last2}
0\overset{\eqref{below3}}{<}\frac{k_1}{n_0 +1}|\{x \in \Omega_{j_0,n_0}: Dw_{n_0 + 1} \in U^3_{I + n_0 + 1}\}| \overset{\eqref{below2}}{\le} |\{x \in \Omega_{j_0,n_0}: Dw_{n_0+ 2} \in U^k_{I + n_0+ 1, I + n_0 + 1}\}|.
\end{equation}
Since
\[
\prod_{i = {n_0+2}}^\infty \left(1-\frac{C}{2^i}\right) > 0,
\]
from \eqref{last1}-\eqref{last2} we infer \eqref{frombelow}, and we conclude the proof of the Theorem.
\end{proof}

\bigskip
\textbf{ Acknowledgements}. 
The authors have been supported by the SNF Grant 182565. 

\bibliographystyle{plain}
\bibliography{Third}
\end{document}